\documentclass[onefignum,onetabnum,oneeqnum,onethmnum,reqno]{siamonline190516}

\usepackage{graphicx}        % standard LaTeX graphics tool
\usepackage{amsmath, amssymb, amsfonts, amsbsy, latexsym, xcolor,dsfont}
\usepackage{amscd,amsxtra}

\usepackage{enumitem}
\setlist[enumerate]{leftmargin=.5in}
\setlist[itemize]{leftmargin=.5in}

\usepackage[T1]{fontenc}
\usepackage[utf8]{inputenc}
\usepackage[english]{babel}
\usepackage{twoopt}
\usepackage{doi}

\usepackage{mathdots}
\usepackage{tikz}
\usetikzlibrary{arrows,%
                petri,%
                topaths}%

\usepackage{color}
\usepackage[normalem]{ulem}
\newcommand{\Covering}{\operatorname{Cov}}
\newcommand{\NNSet}{\mathcal{NN}}
\newcommand{\Lip}{\operatorname{Lip}}
\newcommand{\ClosedBall}{\overline{B}}

\newcommand{\FirstN}[1]{\underline{#1}}
\newcommand{\J}{\boldsymbol{J}}
\newcommand{\boldN}{\boldsymbol{N}}

\usepackage{hyperref}
\hypersetup{
  colorlinks,
  citecolor={blue!50!black},
  urlcolor={blue!80!black},
 pdfview={FitH},
 pdfstartview={FitH},
 bookmarksnumbered={true},
 pdfauthor = {A. Caragea, D.G. Lee, J. Maly, G. Pfander, F. Voigtlaender},
 pdfsubject = {},
 pdfkeywords = {},
 pdfcreator = {},
 pdfproducer = PDFlatex}

\usepackage{silence}
\usepackage[noadjust]{cite}
\usepackage{todonotes}
\usepackage{url}

\newcommand{\CompressedDots}{\makebox[1em][c]{.\hfil.\hfil.}}

\theoremstyle{plain}

\newsiamremark{remark}{Remark}
\numberwithin{equation}{section}

\renewcommand{\Re}{\operatorname{Re}}
\renewcommand{\Im}{\operatorname{Im}}

\newcommand{\din}{d_{\mathrm{in}}}
\newcommand{\dout}{d_{\mathrm{out}}}

\newcommand*{\CC}{\mathbb{C}}
\def\sgn{\mathop{\operatorname{sgn}}}
\makeatletter
\newcommand{\LeftEqNo}{\let\veqno\@@leqno}
\DeclareFontFamily{U}{mathx}{\hyphenchar\font45}
\DeclareFontShape{U}{mathx}{m}{n}{
      <5> <6> <7> <8> <9> <10>
      <10.95> <12> <14.4> <17.28> <20.74> <24.88>
      mathx10
      }{}
\DeclareSymbolFont{mathx}{U}{mathx}{m}{n}
\DeclareFontSubstitution{U}{mathx}{m}{n}
\DeclareMathAccent{\widecheck}{0}{mathx}{"71}
\DeclareMathAccent{\wideparen}{0}{mathx}{"75}

\newcommand{\Indicator}{{\mathds{1}}}

\newcommand{\Sobolev}{\mathcal{W}}
\newcommand{\CalO}{\mathcal{O}}
\newcommand{\CalA}{\mathcal{A}}
\newcommand{\CalF}{\mathcal{F}}
\newcommand{\CalW}{\mathcal{W}}

\DeclareMathOperator{\supp}{supp} %

\newcommand{\eps}{\ensuremath{\varepsilon}}

\newcommand*{\numbersys}[1]{\ensuremath{\mathbb{#1}}}

\newcommand*{\R}{\numbersys{R}}

\newcommand*{\N}{\numbersys{N}}

\newcommand{\abs}[1]{\ensuremath{\left\lvert#1\right\rvert}}

\newcommand{\z}{\boldsymbol{z}}
\newcommand{\w}{\boldsymbol{w}}
\newcommand{\x}{\boldsymbol{x}}
\newcommand{\y}{\boldsymbol{y}}
\newcommand{\approxMult}{\widetilde{\times}}
\newcommand{\m}{\boldsymbol{m}}
\newcommand{\n}{\boldsymbol{n}}

\newcommand{\inn}{\mathrm{in}}
\newcommand{\out}{\mathrm{out}}
\newcommand{\compose}{\bullet}

\binoppenalty=\maxdimen
\relpenalty=\maxdimen
\let\emptyset\varnothing
\headers{Approximation bounds for complex neural networks}
        {A. Caragea, D.G. Lee, J. Maly, G.E. Pfander, and F. Voigtlaender}

\author{A.~Caragea\thanks{KU Eichstätt–Ingolstadt,
Mathematisch–Geographische Fakultät,
Ostenstraße 26,
Kollegiengebäude I Bau B,
85072 Eichstätt,
Germany
(\email{andrei.caragea@ku.de}, \email{daegwans@gmail.com}, \email{johannes.maly@ku.de}, \email{pfander@ku.de}, \mbox{\email{felix.voigtlaender@ku.de}})}
\and D.G.~Lee\footnotemark[2]
\and J.~Maly\footnotemark[2]
\and G.E.~Pfander\footnotemark[2]
\and F.~Voigtlaender\footnotemark[2] \thanks{%
Department of Mathematics,
Technical University of Munich,
85748 Garching bei München,
Germany.}
}

\title{Quantitative approximation results\\ for complex-valued neural networks
\thanks{Version: \today \funding{
A.~Caragea acknowledges support by the DFG Grant PF 450/11-1.
D.G.~Lee acknowledges support by the DFG Grants PF 450/6-1 and PF 450/9-1.
FV acknowledges support by the German Science Foundation (DFG)
in the context of the Emmy Noether junior research group VO 2594/1--1.}}}

\hypersetup{
 pdftitle = {Quantitative approximation results for complex-valued neural networks},
}

\newcommand{\PetersenRef}{PetersenOptimalApproximation}

\begin{document}

\maketitle

\begin{abstract}
  Until recently, applications of neural networks in machine learning have
  almost exclusively relied on real-valued networks.
  It was recently observed, however, that complex-valued neural networks (CVNNs) exhibit
  superior performance in applications in which the input is naturally complex-valued,
  such as MRI fingerprinting.
  While the mathematical theory of real-valued networks has, by now,
  reached some level of maturity, this is far from true for complex-valued networks.
  In this paper, we analyze the expressivity of complex-valued networks
  by providing explicit quantitative error bounds for approximating $C^n$ functions
  on compact subsets of $\mathbb{C}^d$ by complex-valued neural networks that employ
  the modReLU activation function, given by $\sigma(z) = \mathrm{ReLU}(|z| - 1) \, \mathrm{sgn} (z)$,
  which is one of the most popular complex activation functions used in practice.
  We show that the derived approximation rates are optimal (up to log factors)
  in the class of modReLU networks with weights of moderate growth.
\end{abstract}



\section{Introduction}
\label{sec:Intro}

%%%%%%%%%%%%%%%%%%%%%%%%%%%%%%%%%%%%%%%%%%%%%%%%%%%%%%%%%%%%
%\input{Introduction.tex}		%!TEX root=./Draft.tex

Motivated by the remarkable practical success of machine learning algorithms based on
deep neural networks (collectively called \emph{deep learning} \cite{LeCunDeepLearningNature})
in applications like image recognition \cite{KrizhevskyImagenet} and machine translation
\cite{SutskeverMachineTranslation}, the expressive power of such neural networks
is the topic of an active and rich area of study
\cite{YarotskyReLUBounds,YarotskyPhaseDiagram,PetersenOptimalApproximation,LuReLUFiniteDepthUniform}.
Results on the expressivity of real-valued neural networks date back to the 90s,
when the main focus was on networks with \emph{smooth} activation functions
\cite{MhaskarSmoothActivationOptimalRate}.
More recently, emphasis has shifted towards networks
using the \emph{rectified linear unit (ReLU)} activation function $\varrho(x) = \max \{ 0, x \}$,
as those networks have been observed to yield similar expressive power at a
greatly reduced training time cost \cite{LeCunDeepLearningNature,GlorotRectifierNetworks}.

Due to the missing support for complex arithmetic in the leading deep learning software libraries
\cite{TrabelsiDeepComplexNetworks}, practical applications of deep neural networks
have almost exclusively employed \emph{real-valued} neural networks.
Recently, however, there has been an increased interest in \emph{complex-valued neural networks (CVNNs)}
for problems in which the input is naturally complex-valued
and in which a faithful treatment of phase information is important
\cite{LustigComplexNNForMRI,TrabelsiDeepComplexNetworks}.
For instance, for the problem of MRI fingerprinting, CVNNs significantly outperform
their real-valued counterparts \cite{LustigComplexNNForMRI}.
Moreover, CVNNs have demonstrated greatly improved stability and convergence
properties for the setting of recurrent neural networks
\cite{WolterComplexGatedRNNs,BengioUnitaryEvolutionRNN}.

Motivated by the increased interest in complex-valued neural networks,
we herein initiate the analysis of their expressive power,
quantified by their ability to approximate functions of a given regularity.
Specifically, we analyze how well CVNNs with the modReLU activation function
(defined in \Cref{sub:CVNNs}) can approximate functions of Sobolev regularity $\Sobolev^{n,\infty}$
on compact subsets of $\CC^d$ (see \Cref{sub:SmoothnessAssumptions}).
The explicit result is given in \Cref{sub:MainResult}.

\subsection{Complex-valued neural networks and the modReLU function}%
\label{sub:CVNNs}

In a complex-valued neural network (CVNN), each neuron computes a function of the form
$\z \mapsto \sigma(\w^T \z + b)$ with $\z, \w \in \CC^N$ and $b \in \CC$, where
$\sigma : \CC \to \CC$ is a complex activation function.

Formally, a \emph{complex-valued neural network (CVNN)} is a tuple
$\Phi = \big( (A_1,b_1),\dots,(A_L,b_L) \big)$, where $L =: L(\Phi) \in \N$ denotes
the \emph{depth} of the network and where $A_\ell \in \CC^{N_\ell \times N_{\ell-1}}$
and $b_\ell \in \CC^{N_\ell}$ for $\ell \in \{ 1,\dots,L \}$.
Then $\din (\Phi) := N_0$ and $\dout (\Phi) := N_L$ denote the input- and output-dimension of $\Phi$.
Given any function $\sigma : \CC \to \CC$, the \emph{network function} associated to the network
$\Phi$ (also called the \emph{realization} of $\Phi$) is the function
\[
  R_\sigma \Phi
  := T_L \circ (\sigma \circ T_{L-1}) \circ \cdots \circ (\sigma \circ T_1)
  : \CC^{\din(\Phi)} \to \CC^{\dout(\Phi)}
  \quad \text{ where } \quad
  T_\ell \, \z = A_\ell \, \z + b_\ell ,
\]
and where $\sigma$ acts componentwise on vectors, meaning
$\sigma\bigl( (z_1,\dots,z_k) \bigr) = \big( \sigma(z_1),\dots,\sigma(z_k) \big)$.
The functions $T_L$ and $\sigma\circ T_\ell$ for $\ell \in \{1,\dots,L-1\}$
are the functions computed by the different \emph{layers} of the network $\Phi$.
The network $\Phi$ is called \emph{shallow} if $L=2$, i.e., if $\Phi$ has only one
``internal layer'' (neither an input, nor an output layer),
which is usually called a \emph{hidden layer}.

The \emph{number of neurons} $N(\Phi)$ of $\Phi$ is $N(\Phi) := \sum_{\ell=0}^L N_\ell$,
the \emph{width} (or \emph{breadth}) of $\Phi$ is $B(\Phi) := \max_{0 \leq \ell \leq L} N_\ell$,
and the \emph{number of weights} of $\Phi$ is
$W(\Phi) := \sum_{j=1}^L (\| A_j \|_{\ell^0} + \| b_j \|_{\ell^0})$,
where $\| A \|_{\ell^0}$ denotes the number of nonzero entries of a matrix or vector $A$.
Moreoever, writing $\| A \|_\infty := \max_{i,j} |A_{i,j}|$ for a matrix (or vector) $A$,
we define the norm of the network $\Phi$ as
$\| \Phi \| := \max_{1 \leq \ell \leq L} \max \{ \| A_\ell \|_\infty, \| b_\ell \|_\infty \}$.
We then say that \emph{the weights of $\Phi$ are bounded by $C \geq 0$} if
$\| \Phi \| \leq C$.

Finally, we will also use the notion of a \emph{network architecture}%
\footnote{The term ``network architecture'' as used here does \emph{not} refer
to conceptual network architectures like feed-forward networks,
recursive neural networks, and others.
Instead, since we are \emph{only} concerned with fully connected feed-forward networks,
the ``network architecture'' only prescribes the network shape in terms of
the number of layers, the number of neurons per layer, and which weights of the network
may be non-zero.
This terminology is widespread in the literature studying the approximation properties
of neural networks; see e.g.\ \cite{YarotskyReLUBounds,PetersenOptimalApproximation}.}.
Formally, this is a tuple ${\CalA = \big( (N_0,\dots,N_L), (I_1,\dots,I_L), (J_1,\dots,J_L) \big)}$
where $(N_0,\dots,N_L)$ determines the depth $L$ of the network and the number of neurons $N_\ell$
in each layer.
The sets ${J_\ell \subset \{ 1,\dots,N_\ell \}}$
and ${I_\ell \subset \{ 1,\dots,N_\ell \} \times \{ 1,\dots,N_{\ell-1} \}}$
determine which weights of the network are permitted to be nonzero.
Thus, \emph{a network $\Phi$ is of architecture $\CalA$} as above if
${\Phi = \big( (A_1,b_1),\dots,(A_L,b_L) \big)}$ where $A_\ell \in \CC^{N_\ell \times N_{\ell-1}}$
and $b_\ell \in \CC^{N_\ell}$, and if furthermore $(A_\ell)_{j,k} = 0$ if
$(j,k) \notin I_\ell$ and $(b_\ell)_j = 0$ if $j \notin J_\ell$.
The number of weights and neurons of an architecture $\CalA$ are defined as
$W(\CalA) := \sum_{\ell=1}^L (|J_\ell| + |I_\ell|)$ and $N(\CalA) := \sum_{\ell=0}^L N_\ell$,
respectively.

\medskip{}

In the present paper, we focus on neural networks
using the \emph{modReLU} activation function
\begin{equation}
  \sigma : \quad
  \CC \to \CC, \qquad
  z \mapsto \varrho(|z| - 1) \, \sgn(z)
            = \begin{cases}
                0,                 & \text{if } |z| \leq 1, \\
                z - \frac{z}{|z|}, & \text{if } |z| \geq 1
              \end{cases}
  \label{eq:ModReLUDefinition}
\end{equation}
proposed in \cite{BengioUnitaryEvolutionRNN} as a generalization
of the ReLU activation function ${\varrho : \R \to \R, x \mapsto \max \{ 0, x \}}$
to the complex domain.
Note that the complex sign function is defined as $\sgn(z) = \frac{z}{|z|}$
for $z \neq 0$, and $\sgn(z) = 0$ else.
We briefly discuss other activation functions in \Cref{sub:RelatedWork}.

\subsection{Smoothness assumptions}
\label{sub:SmoothnessAssumptions}

We are interested in approximating functions $f : \CC^d \to \CC$
that belong to the Sobolev space $\Sobolev^{n,\infty}$, with differentiability
understood in the sense of real variables.
Specifically, let
\[
  Q_{\CC^d}
  := \big\{
       \z = (z_1,\dots,z_d) \in \CC^d
       :
       \Re{z_k}, \Im{z_k} \in [0,1] \text{ for all } 1 \leq k \leq d
     \big\}
\]
be the (real) unit cube in $\CC^d$.
As in the definition of $Q_{\CC^d}$, we will use throughout the paper boldface characters
to denote real and complex vectors.

Identifying $\z = (z_1,\dots,z_d) \in \CC^d$
with $\x = \bigl(\Re(z_1),\dots,\Re(z_d),\Im(z_1),\dots,\Im(z_d)\bigr)\in\R^{2d}$,
we will consider $\CC^d\cong \R^{2d}$ as usual.
With this in mind, a complex function $g:\CC^d\to\CC$ can be identified
with a pair of functions $g_{\Re},g_{\Im}:\R^{2d}\to\R$
given by $g_{\Re} = \Re(g)$ and $g_{\Im} = \Im(g)$.

Given a real function $f : [0,1]^{2d} \to \R$ and $n \in \N$,
we write $f \in \Sobolev^{n,\infty}([0,1]^{2 d}; \R)$
if $f$ is $n-1$ times continuously differentiable
with all derivatives of order $n-1$ being Lipschitz continuous.
We then define
\[
  \| f \|_{\Sobolev^{n,\infty}}
  := \max
     \Big\{
       \max_{|\alpha| \leq n - 1}
         \| \partial^\alpha f \|_{L^\infty} ,
       \max_{|\alpha| = n-1}
         \mathrm{Lip}(\partial^\alpha f)
     \Big\} .
\]
Using this norm, we define the unit ball in the Sobolev space $\Sobolev^{n,\infty}$ as
\[
  F_{n,d}
  := \big\{
       f \in \Sobolev^{n,\infty}([0,1]^{2d};\R)
       \quad:\quad
       \|f\|_{\Sobolev^{n,\infty}}\leq 1
     \big\}
\]
and define the set of functions that we seek to approximate by
\[
  \mathcal{F}_{n,d}
  := \big\{
       g : Q_{\CC^d} \to \CC
       \quad:\quad
       g_{\Re},g_{\Im} \in F_{n,d}
     \big\}.
\]

\subsection{Main result}%
\label{sub:MainResult}

Our main result provides explicit error bounds for approximating
functions $g \in \CalF_{n,d}$ using modReLU networks.
This result can be seen as a generalization to the complex domain
of the approximation bounds for ReLU networks developed in \cite{YarotskyReLUBounds}.

\begin{theorem}\label{thm:IntroductionMainTheorem}
  For any $d,n \in \N$, there exists $C = C(d,n) > 0$ with the following property:

  Given any $\eps\in(0,1)$ there exists a  \emph{modReLU}-network architecture $\CalA$
  with no more than $C \cdot \ln(2/\eps)$ layers and no more than
  $C \cdot \eps^{-2d/n} \cdot \ln^2(2/\eps)$ weights such that for any
  $g \in \CalF_{n,d}$ there exists a network $\Phi$ of architecture $\CalA$
  with all weights bounded by $C \cdot \eps^{-44 d}$ and such that
  $|g(\z) - R_\sigma \Phi (\z)| \leq \eps$ for all $\z \in Q_{\CC^d}$.
\end{theorem}

The exponent $-\frac{2 d}{n}$ in place of $- \frac{d}{n}$ in the real setting
is a consequence of the identification $\CC^d \cong \R^{2 d}$.
More precisely, making the identifications $\CC^d \cong \R^{2 d}$ and $\CC \cong \R^2$
and using (real-valued) ReLU networks (with two output channels), the results in
\cite{YarotskyReLUBounds} show that---up to logarithmic factors---ReLU networks achieve
\emph{the same} approximation bounds as those shown in \Cref{thm:IntroductionMainTheorem} for modReLU-CVNNs.
Thus, as far as the asymptotic approximation rate is concerned, modReLU-CVNNs
do not \emph{strictly improve} on the approximation capabilities of ReLU networks,
but they can \emph{match} their approximation power.
This is an important theoretical finding, since even though CVNNs were found to have advantages
in several applications \cite{LustigComplexNNForMRI,BengioUnitaryEvolutionRNN},
up to now, no quantitative approximation results for CVNNs were known whatsoever%
---only universal approximation type results were available
\cite{ArenaApproximationCapabilityOfComplexNeuralNetworks,VoigtlaenderCVNNUniversality}.
Our results show that, at least for the approximation problem considered here,
there is no additional ``cost'' in using CVNNs, compared to ReLU networks.

Since \Cref{thm:IntroductionMainTheorem} only provides asymptotic rates
(i.e., no explicit bound on the constant $C$ is provided) and since the $C^n$ assumption
regarding the function to be approximated or learned
cannot be verified in practical applications, the theorem is of limited use
for guiding deep learning practitioners.
Rather, it is intended as a first step towards mathematically understanding
the expressivity of CVNNs and is furthermore expected to be informative for
other theoretical works, for instance for analyzing the performance of CVNNs
for approximating the solutions of PDEs, similar to the results in
\cite{GrohsDNNHighDimensionalElliptic,GrohsDNNHamiltonJacobiBellman,Gonon2019uniform}.

\begin{remark}
  Note that the architecture and therefore the size of the network $\Phi$ is independent
  of the function $g$ to approximate, once we fix an approximation accuracy $\eps$
  and the parameters $n$ and $d$.
  Only the choice of weights depends on $g$.
\end{remark}

\subsection{Comparison to existing work}%
\label{sub:RelatedWork}

\paragraph{Approximation results for CVNNs}

While the approximation properties of real-valued neural networks are comparatively well understood
by now, the corresponding questions for complex-valued networks remain mostly open.
In fact, even the property of universality---well studied for real-valued networks
\cite{HornikUniversalApproximation,HornikStinchcombeWhiteUniversal,
CybenkoUniversalApproximation,PinkusUniversalApproximation}---%
was only settled for very specific activation functions
\cite{ArenaNNInMultidimensionalDomains,ArenaMLPToApproximateComplexValuedFunctions,
ArenaApproximationCapabilityOfComplexNeuralNetworks,HiroseCliffordBook},
until the recent paper \cite{VoigtlaenderCVNNUniversality} resolved the question.
This universal approximation theorem for CVNNs highlights that the properties of complex-valued
networks are significantly more subtle than those of their real-valued counterparts:
Real-valued networks (either shallow or deep) are universal if and only if
the activation function is not a polynomial \cite{PinkusUniversalApproximation}.
In contrast, shallow \emph{complex-valued} networks are universal if and only if
the real part or the imaginary part of the activation function $\sigma$ is not \emph{polyharmonic},
while \emph{deep} complex-valued networks (with more than one hidden layer)
are universal if and only if $\sigma$ is neither holomorphic, nor antiholomorphic,
nor a polynomial (in $z$ and $\overline{z}$).
For instance, \emph{deep} networks with the activation function $\sigma(z) = \overline{z} \cdot e^z$
are universal, but \emph{shallow} networks with this activation function are not.

Aside of these purely qualitative universality results,
\emph{no quantitative approximation bounds for complex-valued networks are known whatsoever}.
The present paper is thus the first to provide such bounds.

\smallskip{}
\paragraph{Role of the activation function}%

As empirically observed in \cite{BengioUnitaryEvolutionRNN,LustigComplexNNForMRI},
the main advantage of complex-valued networks over their real-valued counterparts
stems from the fact that the set of implementable \emph{complex} activation functions
is much richer than in the real-valued case.
In fact, each real-valued activation function $\rho : \R \to \R$
can be lifted to the complex function $\sigma(z) := \rho (\Re z)$;
then, ${\sigma(\w^T \z + b) = \rho(\boldsymbol{\alpha}^T \x - \boldsymbol{\beta}^T \y + \Re b)}$
for $\z = \x + i \y$ and ${\w = \boldsymbol{\alpha} + i \boldsymbol{\beta}}$.
Thus, identifying $\CC^d \cong \R^{2 d}$, every real-valued network
can be written as a complex-valued one.
Therefore, one can in principle transfer every approximation result involving real-valued networks
to a corresponding complex-valued result.
Similar arguments apply to activation functions of the form
${\sigma(z) \!=\! \rho(\Re z) \!+\! i \rho(\Im z)}$.

However, using such ``intrinsically real-valued'' activation functions forfeits
the main benefits of using complex-valued networks,
namely increased expressivity and a faithful handling of phase and magnitude information.
Therefore, the two most prominent complex-valued activation functions appearing in the literature
(see \cite{BengioUnitaryEvolutionRNN,TrabelsiDeepComplexNetworks,LustigComplexNNForMRI})
are the \emph{modReLU} (see \Cref{eq:ModReLUDefinition}) and the \emph{complex cardioid}
(given by $\sigma(z) = \frac{z}{2} \cdot \bigl(1 + \frac{\Re z}{|z|}\bigr)$),
neither of which is of the form $\rho(\Re(z))$ for a real activation function $\rho$.

In the present work, we focus on the modReLU activation function because it satisfies the natural
\emph{phase homogeneity property} $\sigma(e^{i \theta} z) = e^{i \theta} \sigma(z)$.
Investigating the complex cardioid---and other complex-valued activation functions---%
is an interesting topic for future work.

\smallskip{}
\paragraph{Role of the network depth}%

Deep networks greatly outperform their shallow counterparts
in applications \cite{LeCunDeepLearningNature};
therefore, much research has been devoted to rigorously quantify the influence
of the network depth on the expressivity of (real-valued) neural networks.
The precise findings depend on the activation function:
While for \emph{smooth} activation functions, already \emph{shallow} networks
with $\CalO(\eps^{-d/n})$ weights and neurons can uniformly approximate
functions ${f \in C^n([0,1]^d)}$ up to error $\eps$ (see \cite{MhaskarSmoothActivationOptimalRate}),
this is not true for ReLU networks.
To achieve the same approximation rate, ReLU networks need at least $\CalO(1 + \frac{n}{d})$
layers \cite{PetersenOptimalApproximation,SafranShamirICMLPaper,SafranShamirArxivPaper}.
The proofs of these bounds crucially use that the ReLU is piecewise linear.
Since this is not true of the modReLU, these arguments do not apply here.

Regarding sufficiency, the best known approximation result for ReLU networks
\cite{YarotskyReLUBounds} shows---similar to our main theorem---that ReLU networks with
depth $\CalO(\ln(2/\eps))$ and $\CalO(\eps^{-d/n} \ln(2/\eps))$ weights
can approximate functions $f \in C^n([0,1]^d)$ uniformly up to error $\eps$.
For networks with bounded depth, similar results are only known for approximation in $L^p$
\cite{PetersenOptimalApproximation} or for approximation in terms of the network \emph{width}
instead of the number of nonzero weights \cite{LuReLUFiniteDepthUniform}.
It is an interesting question whether these two results extend to modReLU networks as well.

Finally, we mention an intriguing result in \cite{YarotskyPhaseDiagram} which shows that
\emph{extremely deep} ReLU networks (for which the number of layers
is proportional to the number of weights)
with \emph{extremely complicated weights} (meaning the number of significant digits
per weight grows unboundedly as $\eps \downarrow 0$)
can approximate functions $f \in C^n([0,1]^d)$
up to error $\eps$ using only $\CalO(\eps^{-d/(2n)})$ weights (up to log factors).
Due to the prohibitive complexity of the network weights
this bound has limited practical significance,
but is an extremely surprising and insightful mathematical result.
We expect that the arguments in \cite{YarotskyPhaseDiagram} can be extended
to modReLU networks, but leave this as future work.

\smallskip{}
\paragraph{Optimality}%

For modReLU networks with polynomial growth of the individual weights and logarithmic growth
of the depth (as in \Cref{thm:IntroductionMainTheorem}), the approximation rate of
\Cref{thm:IntroductionMainTheorem} is essentially optimal.
We prove this in detail in \Cref{sec:Optimality}, \Cref{thm:Optimality}.
Our proof relies on \emph{entropy arguments}, which are closely related to the
proof techniques based on rate distortion theory as used in
\cite{BoelcskeiSparseNN,PetersenOptimalApproximation}.
Furthermore, for deriving suitable covering bounds for certain network sets
(which then give rise to entropy bounds), we borrow several proof ideas from
\cite{BlackScholesGeneralizationError}.

For ReLU networks, a similar optimality result holds for networks with logarithmic growth
of the depth \emph{even without assumptions on the magnitude of the network weights}
\cite{YarotskyReLUBounds}.
The proof relies on sharp bounds for the VC dimension of ReLU networks \cite{BartlettVCBounds}.
For modReLU networks, a similar question is more subtle,
since to the best of our knowledge no analogous VC dimension bounds are available.
We thus leave it as future work to study  optimality \emph{without} assumptions on the
magnitude of the network weights.

\subsection{Structure of the paper}%
\label{sub:PaperStructure}

Inspired by \cite{YarotskyReLUBounds}, our proof of \Cref{thm:IntroductionMainTheorem} proceeds
by locally approximating $g$ using Taylor polynomials, and then showing that these Taylor polynomials
and a suitable partition of unity can be well approximated by modReLU networks.
To prove this, we first show in \Cref{sec:RealImaginaryPart} that modReLU networks
\emph{of constant size} can approximate the functions $z \mapsto \Re z$ and $z \mapsto \Im z$
arbitrarily well---only the magnitude of the individual weights of the network grows
as the approximation accuracy improves.
Then, based on proof techniques in \cite{YarotskyReLUBounds}, we show in
\Cref{sec:SquareStuff} that modReLU networks with $\CalO(\ln^2(2/\eps))$
weights and $\CalO(\ln(2/\eps))$ layers can approximate the function $z \mapsto (\Re z)^2$
up to error $\eps$.
By a polarization argument, this also allows to approximate the product function
$(z,w) \mapsto z w$; see \Cref{sec:ProductStuff}.
After describing in \Cref{sec:PartitionOfUnity} how a partition of unity
can be implemented with modReLU networks,
we combine all the ingredients in \Cref{sec:MainResult} to prove \Cref{thm:IntroductionMainTheorem}.
Finally, \Cref{sec:Optimality} proves that \Cref{thm:IntroductionMainTheorem} is essentially
optimal.

\section{Approximating real and imaginary parts}%
\label{sec:RealImaginaryPart}

%%%%%%%%%%%%%%%%%%%%%%%%%%%%%%%%%%%%%%%%%%%%%%%%%%%%%%%%%%%%
%\input{RealImaginary.tex}		%!TEX root=./Draft.tex

This section shows that modReLU networks \emph{of constant size}
can approximate the functions $z \mapsto \Re z$ and $z \mapsto \Im z$
arbitrarily well:

\begin{proposition} \label{prop:RealImaginary}
	For any $R \geq 1$ and $\eps \in (0,1)$, there exist
  functions $\Re_{R,\eps}, \Im_{R,\eps} : \CC \to \CC$
  that are implemented by \emph{shallow} $\sigma$-networks
	with $5$ neurons and $10$ weights, all bounded in absolute value
	by $C \cdot R^3 / \eps^3$ with an absolute constant $C > 0$, satisfying
  \[
    |\Re_{R,\eps}(z) - \Re(z)| \leq \eps
    \quad \text{ and } \quad
    |\Im_{R,\eps}(z) - \Im(z)| \leq \eps
    \qquad \text{ for all } \, z \in \CC \text{ with } |z| \leq R.
  \]
\end{proposition}

To prove \Cref{prop:RealImaginary}, we need two ingredients:
First, modReLU networks can implement the identity function
on bounded subsets of $\CC$ \emph{exactly}.
To be precise, for arbitrary $R > 0$ it holds that $\mathrm{Id}_R(z)=z$ for $z \in \CC$
with $|z| \leq R$, where
\begin{equation}
	\mathrm{Id}_R (z)
	:= \sigma \big( 2 z + 2 R + 2 \big) - \sigma(z + R + 1) - (R + 1).
	\label{eq:IdentityRepresentation}
\end{equation}

Indeed, for $w \in \CC$ with $|w| \geq 1$, we have
$\sigma(2 w) - \sigma(w) = 2 w - \frac{2 w}{|2 w|} - (w - \frac{w}{|w|}) = w$.
For $z \in \CC$ with $|z| \leq R$, setting $w = z + R + 1$ so that $|w| \geq 1$
gives $\mathrm{Id}_R(z)=z$.

As the second ingredient, we use the following functions, parameterized by $h > 0$:
\begin{align*}
  \Im_h (z) := \frac{-i}{h^2} \cdot \Bigl(\sgn\bigl(h z + \tfrac{1}{h}\bigr) - 1\Bigr)
  \quad \text{and} \quad
  \Re_h (z) := \frac{1}{h^2} \cdot \Bigl(\sgn\bigl(h z - \tfrac{i}{h}\bigr) + i\Bigr),
  \quad z \in \CC .
\end{align*}

The next lemma shows that these complex-valued functions
well approximate the real-valued functions $\Re$ and $\Im$.
The proof of \Cref{prop:RealImaginary} will then consist of showing that
$\Im_h$ and $\Re_h$ can be implemented by modReLU networks.

\begin{lemma}\label{lem:RealImaginary}
For $z \in \CC$ and $0 < h \leq \frac{1}{2 + 2 |z|}$, we have
  \begin{align}
    \Big|
        \Im (z) - \Im_h(z)
      \Big|
      \leq 2 h |z|
    \qquad \text{and} \qquad
    \Big|
        \Re (z) - \Re_h(z)
      \Big|
      \leq 2 h |z| .
      \label{eq:RealImagApprox}
  \end{align}
\end{lemma}

\begin{proof}
  See \Cref{sub:RealImaginaryPartProof}.
\end{proof}

\begin{proof}[Proof of \Cref{prop:RealImaginary}]
	Set $h := \frac{\eps}{2 + 2 R}$, noting that indeed
	$0 < h \leq \frac{1}{2 + 2 |z|}$ and $h \, |z| \leq \frac{\eps}{2}$ whenever $|z| \leq R$.
	Note that $w := h z - \frac{i}{h}$ satisfies
	$|w| \leq \frac{1}{h} + h \, |z| \leq \frac{2}{h} =: R'$ and
	\({
    |w|
    = |h z - \frac{i}{h}|
    \geq \frac{1}{h} - h \, |z|
    \geq 2 - \frac{1}{2}
    \geq 1 ,
	}\)
	so that $\sgn(w) = w - \sigma(w) = \mathrm{Id}_{R'}(w) - \sigma(w)$,
	with $\mathrm{Id}_{R'}$ as in \Cref{eq:IdentityRepresentation}.
	Putting together the definitions of $h,R',w$ and of $\mathrm{Id}_R$,
  we see that
  \begin{equation}
    \begin{split}
      \Re_h (z)
      & = h^{-2} \cdot \big( \mathrm{Id}_{R'} (w) - \sigma(w) + i \big) \\
      & = h^{-2} \cdot
          \Big(
          \sigma(2 h z + \tfrac{4}{h} + 2 - \tfrac{2 i}{h})
          - \sigma(h z + \tfrac{2}{h} + 1 - \tfrac{i}{h})
          - \sigma(h z - \tfrac{i}{h})
          + i - \tfrac{2}{h} - 1
          \Big) \\
      & =: \Re_{R,\eps}(z)
    \end{split}
    \label{eq:RealPartApproximation}
  \end{equation}
  is implemented by a shallow $\sigma$-network with $5$ neurons and $10$ weights (see Figure~\ref{fig:RealPartNet}),
  where all the weights are bounded by $\frac{4}{h^3} \leq C \, R^3 / {\eps}^3$
  for an absolute constant $C > 0$.
  Finally, \Cref{lem:RealImaginary} shows $|\Re(z) - \Re_{R,\eps}(z)| \leq \eps$
  for all $z \in \CC$ with $|z| \leq R$.
  The claim concerning the approximation of $\Im (z)$ is shown similarly.
\end{proof}

\begin{figure}[htb]
\centering
\scalebox{0.65}{%
\begin{tikzpicture}
\begin{scope}[every node/.style={rectangle,thick,draw}]
\node(In) at (0,0) {$z$};
\node(L1) at (6,0) {$\begin{matrix} N_1 \\ N_2 \\ N_3 \end{matrix}$};
\node(Out) at (12,0) {$\Re_{R,\eps}$};
\end{scope}
\begin{scope}[every node/.style={rectangle},
              every edge/.style={draw=blue, thick, sloped, above, align=center, midway}]
  \path [-to] (In) edge node {$\sigma\circ T_1$} (L1);
  \path [-to] (L1) edge node {$T_2$} (Out);
\end{scope}
\end{tikzpicture}}
\caption{Architecture of the network $\Re_{R,\eps}$, where $T_1(\cdot)=A_1(\cdot)+ b_1$
for $A_1=\big(\tfrac{\eps}{1+R},\tfrac{\eps}{2(1+R)},\tfrac{\eps}{2(1+R)}\big)^T$
and $b_1=\big(\tfrac{(8-4i)(1+R)}{\eps}+2, \tfrac{(4-2i)(1+R)}{\eps}+1,\tfrac{-2i(1+R)}{\eps}\big)^T$,
and $T_2(\cdot)=A_2(\cdot)+b_2$ for $A_2=\tfrac{4(1+R)^2}{\eps^2} (1,-1,-1)$
and $b_2=\tfrac{4(1+R)^2}{\eps^2} \cdot \big(-1-\tfrac{4(1+R)}{\eps}+i\big).$}
\label{fig:RealPartNet}
\end{figure}

\section{Approximating the squared real part}
\label{sec:SquareStuff}

The main result of this section is \Cref{prop:RealSquared} below, showing that
the function $z \mapsto (\Re(z))^2$
on the set $\{ z \in \CC : |z| \le R, \ | \Re(z) | \leq 1 \}$
can be uniformly approximated up to error $\eps$ by modReLU networks
with $\CalO \big( \ln( 2 / \eps ) \big)$ layers
and $\CalO \big( \ln^2( 2 / \eps ) \big)$ weights
of size $\CalO \big( R^6 \, \eps^{-7} \big)$.

As a first step towards proving \Cref{prop:RealSquared},
we show that modReLU networks can approximate functions
of the form $z \mapsto \varrho(\Re(z) + c)$ with the usual ReLU $\varrho$;
this will then allow us to use the approximation of the square function by ReLU
networks as derived in \cite{YarotskyReLUBounds}.

\begin{proposition}\label{prop:ReLUreal}
	For any choice of $R \geq 1$, $c \in \R$, and $\eps \in (0,1)$,
  there exist functions $\varrho^{\Re,c}_{R,\eps}, \varrho^{\Im,c}_{R,\eps} : \CC \to \CC$
  that are implemented by depth-$3$ $\sigma$-networks with $6$ neurons and $11$ weights,
  all bounded in absolute value by $C \cdot R^3 / \eps^3 + 2 |c|$ (with an absolute constant $C$),
  satisfying $|\varrho^{\Re,c}_{R,\eps}(z) - \varrho(\Re(z) + c)| \leq \eps$
  and ${|\varrho^{\Im,c}_{R,\eps}(z) - \varrho(\Im(z) + c)| \leq \eps}$
  for all $z \in \CC$ with $|z| \leq R$.
\end{proposition}

\begin{proof}
	Let us first prove the statement for $\varrho^{\Re,c}_{R,\eps}$.
  To this end, first note that the modReLU $\sigma : \CC \to \CC$ is $1$-Lipschitz;
  see \Cref{lem:SigmaLipschitz} below.

  \begin{figure}[ht]
    \begin{center}
      \includegraphics[width=0.36\textwidth]{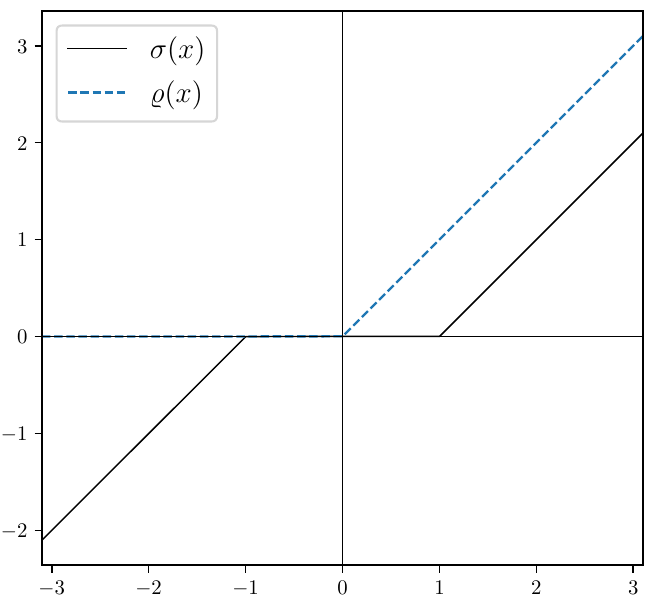}
    \end{center}
    %\captionsetup{width=.8\linewidth,font=small}
    \caption{\label{fig:SigmaRealPlot}A plot of the modReLU function $\sigma$
             on $[-3,3]$.
             The plot shows that $\sigma(x+1) = \varrho(x)$ for $x \in [-2,\infty)$.}
  \end{figure}

	Now, set $\delta := 1 / \bigl(2 \cdot (R+|c|)\bigr)$ and define
	\begin{equation}
	   \varrho_{R,\eps}^{\Re,c} : \quad
     \CC \to \CC, \quad
     z \mapsto \frac{1}{\delta}
               \cdot \sigma \big( 1 + \delta \cdot \big( \Re_{R,\eps} (z)+c \big) \big),
    \label{eq:ComplexifiedReLUNetwork}
	\end{equation}
	where $\Re_{R,\eps}$ is as in \Cref{prop:RealImaginary}.
  Now, a direct computation (see also \Cref{fig:SigmaRealPlot})
  shows that $\sigma(x+1) = \varrho(x)$ for $x \in [-2,\infty)$.
	Because of $|\delta \cdot (\Re(z)+c)| \leq \delta \cdot (|z| + |c|) \leq \frac{1}{2}$
  for $|z| \leq R$, this implies
  \(
    \frac{1}{\delta}
    \cdot \sigma \big( 1 + \delta \cdot (\Re(z) + c) \big)
    = \frac{1}{\delta}
      \cdot \varrho \bigl( \delta \cdot (\Re(z) + c) \bigr)
    = \varrho \big( \Re(z) + c \big) .
  \)
  Combined with the $1$-Lipschitz continuity of $\sigma$, we thus see
	\begin{align*}
    |\varrho(\Re(z) + c) - \varrho_{R,\eps}^{\Re,c}(z)|
    & = \abs{
          \tfrac{1}{\delta} \cdot \sigma ( 1 + \delta \cdot (\Re (z) + c) )
          - \tfrac{1}{\delta} \cdot \sigma ( 1 + \delta \cdot (\Re_{R,\eps} (z) + c) )
        } \\
    &\leq |\Re (z) - \Re_{R,\eps}(z)|
    \leq \eps
    \quad \text{ for all } z \in \CC \text{ with } |z| \leq R .
	\end{align*}
  Based on the properties of $\Re_{R,\eps}$ from \Cref{prop:RealImaginary}
  (see also \Cref{eq:RealPartApproximation} noting that $h = \frac{\eps}{2 + 2 R}$
  in that equation), it follows that $\varrho_{R,\eps}^{\Re,c}$ is implemented
  by a depth-$3$ $\sigma$-network with $6$ neurons and $11$ weights (see Figure~\ref{fig:RealReLUNet}),
  all bounded in absolute value by $2 \, |c| + C \cdot R^3/\eps^3$.
  The construction of $\varrho_{R,\eps}^{\Im,c}$ is similar,
	replacing $\Re_{R,\eps}(z)$ with $\Im_{R,\eps}(z)$.
\end{proof}

\begin{figure}[htb]
\centering
\scalebox{0.65}{%
\begin{tikzpicture}
\begin{scope}[every node/.style={rectangle,thick,draw}]
\node(In) at (0,0) {$z$};
\node(L1) at (6,0) {$\begin{matrix} N_1 \\ N_2 \\ N_3 \end{matrix}$};
\node(L2) at (12,0) {$N_4$};
\node(Out) at (18,0) {$\varrho_{R,\eps}^{\Re,c}(z)$};
\end{scope}
\begin{scope}[every node/.style={rectangle},
              every edge/.style={draw=blue, thick, sloped, above, align=center, midway}]
  \path [-to] (In) edge node {$\sigma\circ T_1$} (L1);
  \path [-to] (L1) edge node {$\sigma\circ T_2$} (L2);
  \path [-to] (L2) edge node {$T_3$} (Out);
\end{scope}
\end{tikzpicture}}
\caption{Architecture of the network $\varrho_{R,\eps}^{\Re,c}$, where $T_1(\cdot)=A_1(\cdot)+ b_1$
for $A_1=\big(\tfrac{\eps}{1+R},\tfrac{\eps}{2(1+R)},\tfrac{\eps}{2(1+R)}\big)^T$
and $b_1=\big(\tfrac{(8-4i)(1+R)}{\eps}+2, \tfrac{(4-2i)(1+R)}{\eps}+1,\tfrac{-2i(1+R)}{\eps}\big)^T$,
and $T_2(\cdot)=A_2(\cdot)+b_2$ for $A_2=\tfrac{2(1+R)^2}{(R+|c|)\eps^2} (1,-1,-1)$
and $b_2=\tfrac{2(1+R)^2}{(R+|c|)\eps^2}\big(-1-\tfrac{4(1+R)}{\eps}+i\big)+1+\tfrac{c}{2(R+|c|)}$,
and finally $T_3(z) = 2(R+|c|) \cdot z.$}
\label{fig:RealReLUNet}
\end{figure}

The next lemma shows that $\sigma : \CC \to \CC$
is $1$-Lipschitz, which was used in the proof above.

\begin{lemma}\label{lem:SigmaLipschitz}
  The modReLU function $\sigma : \CC \to \CC$ defined in \Cref{eq:ModReLUDefinition}
  is $1$-Lipschitz, i.e., $|\sigma(z) - \sigma(w)| \leq |z - w|$ for all $z,w \in \CC$.
\end{lemma}

\begin{proof}
	Simply note that
	\begin{align*}
		\big| \sigma(z) - \sigma(w) \big|
		= \begin{cases}
        0 & \text{if } |z|,|w| \le 1, \\
        |\sigma(z)| = |z|-1 \le |z| - |w| \le |z - w|
        & \text{if } |z| > 1 ~ \text{and} ~ |w| \le 1 , \\
        |\sigma(w)| = |w|-1 \le |w| - |z| \le |w - z|
        & \text{if } |z| \le 1 ~ \text{and} ~ |w| > 1 , \\
        \big| (z - \frac{z}{|z|}) - (w - \frac{w}{|w|}) \big| \le |z - w|
        &  \text{if } |z|,|w|  >  1 ,
      \end{cases}
	\end{align*}
  where we used that if $z, w \in \CC$ with $|z|,|w|  >  1$, then
  \begin{align*}
    \big| (z - \tfrac{z}{|z|}) - (w - \tfrac{w}{|w|}) \big|^2
    &= \big| \tfrac{z}{|z|} \cdot (|z| - 1) - \tfrac{w}{|w|} \cdot (|w| - 1) \big|^2 \\
    &= (|z|-1)^2
       + (|w|-1)^2
       - 2 (|z|-1)(|w|-1) \Re \big( \tfrac{z}{|z|} \, \tfrac{\overline{w}}{|w|} \big)  \\
    &= |z|^2
       + |w|^2
       - 2|z||w| \Re \big(  \tfrac{z}{|z|} \, \tfrac{\overline{w}}{|w|} \big)
       - \big( 2|z| + 2|w| - 2 \big)
         \big( 1- \Re \big(  \tfrac{z}{|z|} \, \tfrac{\overline{w}}{|w|} \big) \big) \\
    &\leq |z|^2
          + |w|^2
          - 2|z||w|  \Re \big(  \tfrac{z}{|z|} \, \tfrac{\overline{w}}{|w|} \big)
      = |z - w|^2 .
      %\qedhere
  \end{align*}
\end{proof}

Our next goal is to construct $\sigma$-networks approximating the function $z \mapsto (\Re z)^2$.
This will be based on combining \Cref{prop:ReLUreal} with the approximation
of the real function $x \mapsto x^2$ by ReLU networks, as presented in \cite{YarotskyReLUBounds}.

The construction in \cite{YarotskyReLUBounds} is based on the following auxiliary functions,
depicted in \Cref{fig:GPlot}:
\begin{alignat*}{5}
  g : & \quad \R \to \R,  \quad
  g(x)  &&:= 2 \varrho(x) - 4 \varrho(x - \tfrac{1}{2}) + 2 \varrho(x - 1) ,  \\
  g_k : & \quad \R \to \R, \quad
  g_k(x) &&:= \underset{k-\text{times}}
                      { \underbrace{ g \circ \cdots \circ g } } (x)
         \quad \text{for} \;\; k \in \N ,  \\
  f_m : & \quad \R \to \R, \quad
  f_m(x) &&:= x - \sum_{k = 1}^m \frac{g_k(x)}{2^{2k}}
         \quad \text{for} \;\; m \in \N \cup \{ 0 \} .
\end{alignat*}

\begin{figure}[h]
  \begin{center}
    \includegraphics[width=0.48\columnwidth]{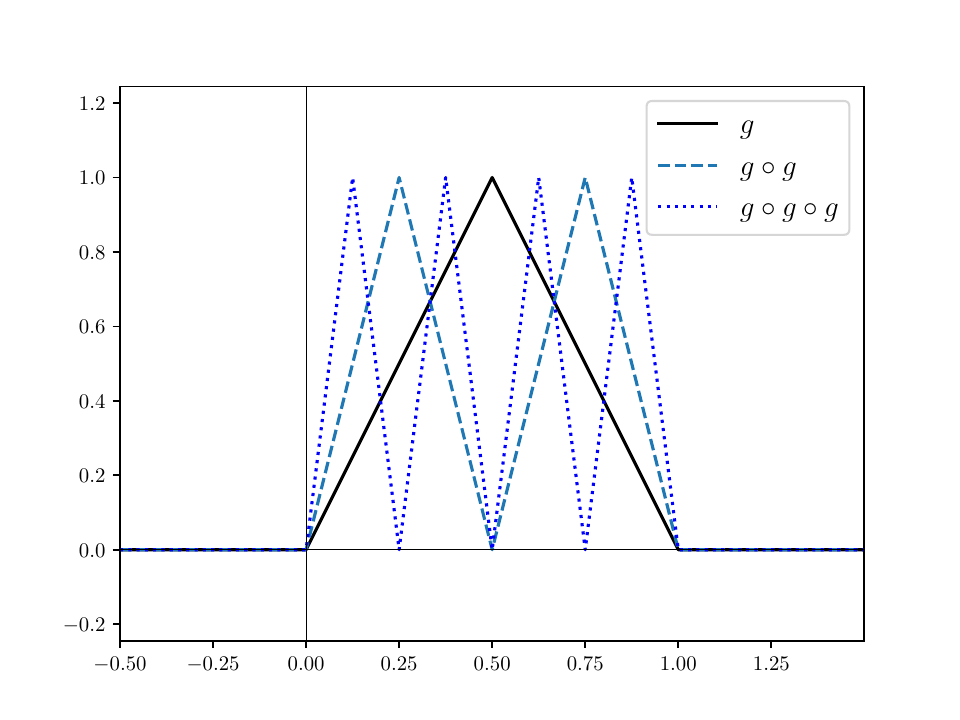}
    \quad
    \includegraphics[width=0.48\columnwidth]{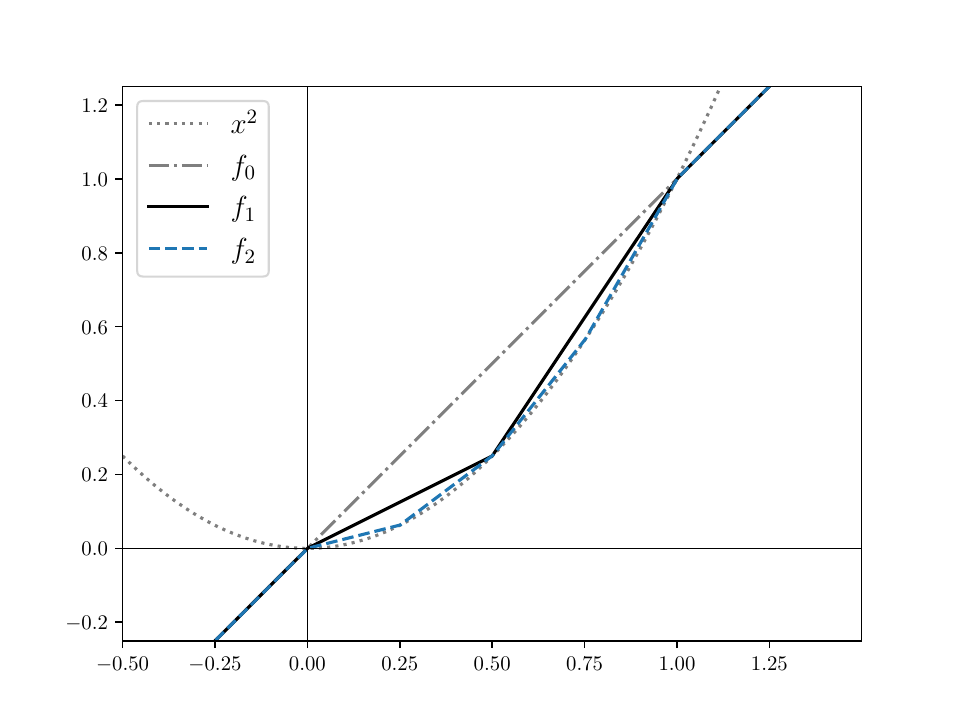}
  \end{center}
  %\captionsetup{width=.8\linewidth,font=small}
  \caption{\label{fig:GPlot}A plot of the function $g$, its compositions
    $g \circ g$ and $g \circ g \circ g$, and the square approximations $f_m$, for $m=0,1,2$.
    Evidently, $g$ is $2$-Lipschitz.}
\end{figure}

\noindent
One can show (cf.~\cite[Proof of Proposition~2]{YarotskyReLUBounds}) that
\begin{align}\label{eq:Estimate1}
  \abs{ x^2 - f_m(x)} \le 2^{-2m-2}
  \quad \text{for} \;\; 0 \leq x \leq 1 .
\end{align}
Further, we define
\begin{alignat*}{5}
  g^{\Re} : & \quad \CC \to \CC,  \quad
  g^{\Re}(z) && := g\bigl(\Re(z)\bigr)
                    	= 2 \varrho\bigl(\Re(z)\bigr)
                      	- 4 \varrho\bigl(\Re(z) - \tfrac{1}{2}\bigr)
                      	+ 2 \varrho\bigl(\Re(z) - 1\bigr)   , \\
  g^{\Re, k} : & \quad \CC \to \CC, \quad
  g^{\Re, k}(z)  && := \underset{k-\text{times}}
                               { \underbrace{ g^{\Re} \circ \cdots \circ g^{\Re} } } (z)
                    = g_k \bigl(\Re (z)\bigr)
    \quad (\text{since} \;\; g : \R \rightarrow \R) .
\end{alignat*}
As is clear from \Cref{fig:GPlot}, the function $g : \R \rightarrow \R$ is $2$-Lipschitz,
which in turn implies that $g^{\Re} : \CC \rightarrow \CC$ is $2$-Lipschitz; indeed,
\(
  \big| g(\Re(z)) - g(\Re(z')) \big|
  \leq 2 \, | \Re(z) - \Re(z') |
  \leq 2 \, | z - z' |
\)
for $z, z' \in \CC$.
In view of \Cref{prop:ReLUreal}, we consider the approximation
of $g^{\Re}$ and $g^{\Re, k}$ respectively by the following functions:
\begin{alignat*}{5}
  g_{R,\eps}^{\Re} :
  & \quad \CC \to \CC, \quad
  g_{R,\eps}^{\Re}(z)
  && := 2 \varrho_{R,\eps}^{\Re,0} (z)
        - 4 \varrho_{R,\eps}^{\Re,-1/2} (z)
        + 2 \varrho_{R,\eps}^{\Re,-1} (z) ,  \\
  g_{R,\eps}^{\Re, k} :
  & \quad \CC \to \CC, \quad
  g_{R,\eps}^{\Re, k}(z)
  && := \underset{k-\text{times}}
                 { \underbrace{ g_{R,\eps}^{\Re} \circ \cdots \circ g_{R,\eps}^{\Re} } } (z)
     \quad \text{for } k \in \N ,
\end{alignat*}
where $R \geq 1$ and $\eps \in (0,1)$.
As the last preparation for the proof of \Cref{prop:RealSquared},
we need the following technical lemma concerning the size of $g_{R,\eps}^{\Re,k} (z)$.

\begin{lemma}\label{lem:MagnitudeBound}
	Let $R \ge 1$, $0 < \eps < \min \{1, \frac{R}{8} \}$, and $k \in \N$.
  Then $\big| g_{R+1,\eps}^{\Re, k}(z) \big| \le R+1$ for all $z \in \CC$ with $|z| \le R+1$.
\end{lemma}

\begin{proof}
	\Cref{prop:ReLUreal} implies that if $z \in \CC$ with $|z| \le R + 1$, then
	\[
    \abs{ g_{R+1,\eps}^{\Re}(z) - g^{\Re}(z) }
    \leq (2 + 4 + 2) \cdot \eps
    \leq R
	\]
	and since $|g^{\Re}(z)| = |g(\Re(z))| \leq 1$ for all $z \in \CC$ (see \Cref{fig:GPlot}), we have
	\[
    |g_{R+1,\eps}^{\Re}(z)|
    \leq  \abs{ g_{R+1,\eps}^{\Re}(z) - g^{\Re}(z) } + \abs{ g^{\Re}(z) }
    \leq R+1 .
	\]
	This shows that $g_{R+1,\eps}^{\Re}$ maps $\{ z \in \CC \colon |z| \leq R+1 \}$ into itself.
	It then follows by induction that $g_{R+1,\eps}^{\Re,k}$ maps
	$\{ z \in \CC \colon |z| \leq R+1 \}$ into itself, as claimed.
\end{proof}

\begin{proposition}\label{prop:RealSquared}
  Let $R \geq 3$ and $0 < \varepsilon < \min\{1,\frac{R}{8}\}$.
  There exists a function $\Phi_{R,\eps} : \CC \to \CC$ that is implemented
  by a $\sigma$-network of depth $\CalO(\ln(2/\eps))$ and width $\CalO(\ln(2/\eps))$
  and with the number of weights and neurons bounded by $\CalO(\ln^2(2/\eps))$
  and all weights bounded by $\CalO(R^6 / \eps^7)$
  and such that $|(\Re z)^2 - \Phi_{R,\eps}(z)| \leq \eps$
  for all $z \in \CC$ with $|z| \leq R$ and $|\Re z| \leq 1$.
\end{proposition}

\begin{proof}
First, it holds for any $z \in \CC$ with $|z| \le R+1$ that
\begin{align}\label{eq:gBound2}
   \abs{ g^{\Re}(z) - g_{R+1,\eps}^{\Re}(z)}
   \leq (2 + 4 + 2) \cdot \eps
   =    8\eps
\end{align}
and $\big| g_{R+1,\eps}^{\Re, k}(z) \big| \le R+1$ for all $k \in \N$,
by \Cref{prop:ReLUreal} and \Cref{lem:MagnitudeBound} respectively.
We claim that this implies
\begin{equation}
  |g^{\Re,k}(z) - g_{R+1,\eps}^{\Re,k}(z)|
  \leq 8 \eps \cdot (2^k - 1)
  \qquad \forall \, k \in \N \text{ and } |z| \leq R+1.
  \label{eq:RealSquaredInductionClaim}
\end{equation}
Indeed, for $k = 1$ we have $g_{R+1,\eps}^{\Re,k} = g^{\Re}_{R+1,\eps}$
and $g^{\Re,k} = g^{\Re}$, so that \Cref{eq:gBound2} shows
$|g^{\Re,k}(z) - g_{R+1,\eps}^{\Re,k}(z)| \leq 8 \eps = 8 \eps \cdot (2^k - 1)$.

Next, suppose \Cref{eq:RealSquaredInductionClaim} holds for some $k \in \N$.
\Cref{lem:MagnitudeBound} shows that $w := g_{R+1,\eps}^{\Re,k}(z)$ satisfies $|w| \leq R+1$.
Further, setting $w' := g^{\Re,k}(z)$, \Cref{eq:RealSquaredInductionClaim} shows
$|w - w'| \leq 8 \eps \cdot (2^k - 1)$. Thus, using that $g^{\Re}$ is $2$-Lipschitz, we see
\begin{align*}
  |g^{\Re,k+1}(z) - g_{R+1,\eps}^{\Re,k+1}(z)|
  & = |g^{\Re}(w') - g_{R+1,\eps}^{\Re}(w)| \\
  & \leq |g^{\Re}(w') - g^{\Re}(w)| + |g^{\Re}(w) - g_{R+1,\eps}^{\Re} (w)| \\
  & \overset{(\ast)}{\leq} 2 \cdot 8\eps \cdot (2^k - 1) + 8 \eps
    = 8 \eps \cdot (2^{k+1} - 1)
\end{align*}
where \Cref{eq:gBound2} was used at $(\ast)$.
Thus, \Cref{eq:RealSquaredInductionClaim} holds for $k+1$ if it holds for $k$.

Now, using the function $\mathrm{Id}_R$ from \Cref{eq:IdentityRepresentation}
(which is implemented by a $2$-layer $\sigma$-network with $7$ weights, all bounded by $2 R + 2$)
and the function $\Re_{R,\eps}$ from \Cref{prop:RealImaginary},
we define for $m \in \N$,
\begin{align*}
  f_{m,R,\eps}(z)
  &:= \Re_{R,\eps} \circ \mathrm{Id}_{R} \circ \cdots \circ \mathrm{Id}_{R}  (z)
      - \sum_{k = 1}^m
          \frac{g_{R+1,\eps}^{\Re, k} \circ \mathrm{Id}_{R} \circ \cdots \circ \mathrm{Id}_{R} (z)}
               {4^k}
	\quad \text{for} \;\; z \in \CC ,
\end{align*}
where the number of the ``factors'' $\mathrm{Id}_{R}$ is chosen such that all (sub)networks
have the same depth and thus can be added/subtracted---%
see \Cref{sub:DNNCalculusComposition,sub:DNNCalculusAddition} for details on implementing
composition and summation of networks.
It then follows for $m \in \N$ and $|z| \le R$ that
\begin{equation}\label{eq:Estimate2}
	\begin{split}
    |f_m(\Re(z)) - f_{m,R,\eps}(z)|
    &\le |\Re(z) - \Re_{R,\eps}(z)|
         + \sum_{k=1}^m
             \frac{| g^{\Re, k} (z) - g_{R+1,\eps}^{\Re, k} (z) |}
                  {4^k} \\
    &\le \eps + 8 \eps \sum_{k=1}^m 2^{-k}
    \le 9 \eps.
	\end{split}
\end{equation}
Setting $m := \big\lceil \tfrac{1}{2} \ln(\tfrac{1}{\eps}) / \ln (2) \big\rceil \in \N$
(so that $2^{-2m-2} \leq \eps$) and combining \eqref{eq:Estimate1} and \eqref{eq:Estimate2},
we deduce for $|z| \le R$ with $0 \leq \Re(z) \leq 1$ that
\begin{align}\label{eqn:FirstPartForRealz-in-01}
  |\Re(z)^2 - f_{m,R,\eps}(z)|
  \le |\Re(z)^2 - f_{m}(\Re(z))| + |f_{m}(\Re(z)) - f_{m,R,\eps}(z)| \le 10 \eps.
\end{align}

\medskip{}

We will now extend this result to $z \in \CC$ with $|z| \leq R$ and $| \Re(z) | \leq 1$.
Given such a $z$, define $w := \frac{1}{2} (z+1)$, noting that $|w| \leq R$ (since $R \geq 1$)
and $0 \leq \Re w \leq 1$.
Therefore, applying \Cref{eqn:FirstPartForRealz-in-01} to $w$ instead of $z$, we see
\(
  |(\Re w)^2 - f_{m,R,\eps}(w)|
  \leq 10 \eps .
\)
Note that
$(\Re w)^2 = \frac{1}{4} (1 + \Re z)^2 = \frac{1}{4} + \frac{1}{2} \Re z + \frac{1}{4} (\Re z)^2$
and hence $(\Re z)^2 = 4 \, (\Re w)^2 - 2 \Re z - 1$.
Thus, setting
\[
  h_{m,R,\eps}(z)
  := 4 \, f_{m,R,\eps}\bigl(\tfrac{1}{2}(z+1)\bigr)
    - 2 \Re_{R,\eps} \circ \mathrm{Id}_R \circ \cdots \circ \mathrm{Id}_R (z)
    - 1
  ,
\]
where again $\mathrm{Id}_R$ is used to match the depth of the (sub)networks, we see
\[
  |(\Re z)^2 - h_{m,R,\eps}(z)|
  \leq 4 \cdot |(\Re w)^2 - f_{m,R,\eps}(w)|
       + 2 \cdot |\Re(z) - \Re_{R,\eps}(z)|
  \leq 42 \eps .
\]

It remains to bound the depth, width, and number of weights
of the $\sigma$-network defining the function $\Phi_{R,\eps} (z) := h_{m,R,\eps} (z)$,
and to estimate the size of the weights.
The following estimates regarding these quantities should be fairly intuitive;
the reader interested in the full details is referred to
\Cref{sub:DNNCalculusComposition,sub:DNNCalculusAddition}.
Note that $f_{m,R,\eps}$, with our choice of
\(
  m
  = \big\lceil \tfrac{1}{2} \ln(\tfrac{1}{\eps}) / \ln (2) \big\rceil
  = \CalO(\ln( 2 /  \eps ))
  ,
\)
is a $\sigma$-network with depth and width $\mathcal{O}(m)$, and with $\CalO (m^2)$ neurons
and weights, all of which are bounded by $\CalO(\ln(2/\eps) R^6 / \eps^6) \subset \CalO(R^6 / \eps^7)$.
Consequently, $\Phi_{R,\eps}$ is a $\sigma$-network
whose depth and width is $\CalO ( \ln( 2/\eps ) )$,
whose number of weights, and neurons are $\CalO ( \ln^2( 2/\eps ) )$,
and whose weights are bounded  by $\CalO ( R^6 / \eps^7 )$.
\end{proof}

\section{Approximating the product of complex numbers}
\label{sec:ProductStuff}

%%%%%%%%%%%%%%%%%%%%%%%%%%%%%%%%%%%%%%%%%%%%%%%%%%%%%%%%%%%%
%\input{ProductStuff.tex}		%!TEX root=./Draft.tex

In this section, we approximate the map
$\CC^2 \rightarrow \CC, (z,w) \mapsto z w$ using modReLU networks.
To do so, we first approximate the function ${\CC^2 \rightarrow \CC, (z,w) \mapsto \Re(z)\Re(w)}$
based on the approximation of $(\Re z)^2$ from \Cref{prop:RealSquared}
and then use a polarization argument.
This idea is motivated by \cite[Proposition~3]{YarotskyReLUBounds}.

\begin{proposition}\label{prop:RealApprox}
	Given $R \ge 3$ and $\varepsilon \in (0,1)$, there is a function
  $\widetilde{\times}_{\Re, R, \eps} \colon \mathbb{C}^2 \rightarrow \mathbb{C}$ such that
  \begin{enumerate}[itemsep=0.2cm, topsep=0.2cm]
		\item for any inputs $z,w \in \mathbb{C}$ with $|z|,|w| \le R$
          we have $| \widetilde{\times}_{\Re, R, \eps} (z,w) - \Re(z)\Re(w)| \le \varepsilon$;

    \item the function $\widetilde{\times}_{\Re, R, \eps}$ is implemented by a $\sigma$-network
          of depth and width $\CalO \big( \ln( R^2 \, \eps^{-1} ) \big)$, with at most
          $\CalO \big( \ln^2( R^2 \, \eps^{-1} ) \big)$ weights and neurons,
          and all weights bounded in absolute value by $\CalO \big( R^{16} \eps^{-7} \big)$.
	\end{enumerate}
\end{proposition}

\begin{proof}
Define $R' := 3$ and note that $0 < \eps' := \frac{\eps}{6R^2} < \frac{1}{54} \leq \min\{1,\frac{R'}{8}\}$.
Therefore, we can apply \Cref{prop:RealSquared} with $R',\eps'$ instead of $R,\eps$,
which produces the function $\Phi_{R',\eps'} = \Phi_{3, \frac{\eps}{6 R^2}}$.
We then set
\[
  \widetilde{\times}_{\Re, R, \eps} (z,w)
  := 2R^2 \cdot
    \Big(
      \Phi_{3,\frac{\eps}{6R^2}} ( \tfrac{z+w}{2R} )
      - \Phi_{3,\frac{\eps}{6R^2}} ( \tfrac{z}{2R} )
      - \Phi_{3,\frac{\eps}{6R^2}} ( \tfrac{w}{2R} )
    \Big)
  \quad \text{for} \;\; z, w \in \CC.
\]
Comparing with the equation
\[
  \Re(z)\Re(w)
  = 2R^2
    \cdot \Big(
            [\Re( \tfrac{z+w}{2R} )]^2
            - [\Re( \tfrac{z}{2R} )]^2
            - [\Re( \tfrac{w}{2R} )]^2
          \Big)
  \quad \text{for} \;\; z, w \in \CC
\]
and applying \Cref{prop:RealSquared}, we see that if $z,w \in \CC$ with $|z|,|w| \le R$, then
\begin{align*}
  & \big| \Re(z)\Re(w) -   \widetilde{\times}_{\Re, R, \eps} (z,w)  \big|  \\
  & \leq 2R^2 \cdot \!
         \Big(
           \Big| [\Re( \tfrac{z+w}{2R} )]^2 - \Phi_{3,\frac{\eps}{6R^2}} ( \tfrac{z+w}{2R} )  \Big|
           \!+\! \Big| [\Re( \tfrac{z}{2R} )]^2 -  \Phi_{3,\frac{\eps}{6R^2}} ( \tfrac{z}{2R} ) \Big|
           \!+\! \Big| [\Re( \tfrac{w}{2R} )]^2 -  \Phi_{3,\frac{\eps}{6R^2}} ( \tfrac{w}{2R} )  \Big|
         \Big) \\
  &\le 2R^2 \cdot \big( \tfrac{\eps}{6R^2} + \tfrac{\eps}{6R^2} + \tfrac{\eps}{6R^2} \big)
  = \eps .
\end{align*}
According to \Cref{prop:RealSquared}, the function $\Phi_{R',\eps'}$
is implemented by a $\sigma$-network of depth and width
$\CalO \big( \ln( \tfrac{2}{\eps'} ) \big)$, with
$\CalO \big( \ln^2( \tfrac{2}{\eps'} ) \big)$ weights and neurons,
and all weights bounded by $\CalO ( (R')^6 (\eps')^{-7} ) )$.
Consequently, the function $\Phi_{3,\frac{\eps}{6R^2}}$
is implemented by a $\sigma$-network of depth and width
$\CalO \big( \ln( 12 R^2 \eps^{-1} ) \big)$, with
$\CalO \big( \ln^2( 12 R^2 \eps^{-1} ) \big)$ weights and neurons,
and all weights bounded by $\CalO ( R^{14} \eps^{-7} )$.
Note that $\widetilde{\times}_{\Re, R, \eps} (z,w)$ is a parallel connection
of three copies of $\Phi_{3,\frac{\eps}{6R^2}}$ with the adjustment
that all weights in the last layer are scaled by a factor of $2 R^2$,
and the first layer is composed with appropriate linear transforms.
Hence, the function $\widetilde{\times}_{\Re, R, \eps} (z,w)$ is again implemented by
a $\sigma$-network whose depth and width are $\CalO \big( \ln( R^2 \, \eps^{-1} ) \big)$,
whose number of weights and neurons are $\CalO \big( \ln^2( R^2 \, \eps^{-1} ) \big)$,
and whose weights are bounded by $\CalO \big( R^{16} \eps^{-7} \big)$ in absolute value.
\end{proof}

As a direct consequence of \Cref{prop:RealApprox},
we obtain an approximation for the complex product function
$\CC^2 \rightarrow \CC, (z,w) \mapsto z w$.

\begin{corollary}\label{cor:ProductApprox}
	Given $R \ge 3$ and $\varepsilon \in (0,1)$, there is a function
  $\widetilde{\times}_{R, \eps} \colon \mathbb{C}^2 \rightarrow \mathbb{C}$ such that
  \begin{enumerate}[itemsep=0.2cm,topsep=0.2cm]
		\item for any inputs $z,w \in \mathbb{C}$ with $|z|,|w| \le R$
          we have $| \widetilde{\times}_{R, \eps} (z,w) - zw| \le \varepsilon$;

    \item the function $\widetilde{\times}_{R, \eps}$ is implemented by a $\sigma$-network
          of depth and width $\CalO \big( \ln( R^2 \, \eps^{-1} ) \big)$,
          with at most $\CalO \big( \ln^2( R^2 \, \eps^{-1} ) \big)$
          weights and neurons, and all weights bounded in absolute value by
          $\CalO \big( R^{16} \eps^{-7} \big)$.
	\end{enumerate}
\end{corollary}

\begin{proof}
  Noting that for $z, w \in \CC$,
    \begin{align*}
      zw
      &=  \Re(z)\Re(w) - \Im(z)\Im(w) + i \big( \Re(z)\Im(w) + \Im(z)\Re(w) \big) \\
      &= \Re(z)\Re(w) - \Re(-iz)\Re(-iw) + i \big( \Re(z)\Re(-iw) + \Re(-iz)\Re(w) \big) ,
    \end{align*}
  we define
  \[
    \widetilde{\times}_{R, \eps} (z,w)
    := \widetilde{\times}_{\Re, R, \frac{\eps}{4}} (z,w)
       - \widetilde{\times}_{\Re, R, \frac{\eps}{4}} (-iz,-iw)
       + i \big(
             \widetilde{\times}_{\Re, R, \frac{\eps}{4}} (z,-iw)
             + \widetilde{\times}_{\Re, R, \frac{\eps}{4}} (-iz,w)
           \big)
    .
  \]
  It then follows from \Cref{prop:RealApprox} that
  $| \widetilde{\times}_{R, \eps} (z,w) - zw| \leq  \varepsilon$
  for all $z,w \in \CC$ with $|z|,|w| \leq R$.
  The function $\widetilde{\times}_{R, \eps} (z,w)$ is a sum
  of four equivalent copies of $\widetilde{\times}_{\Re, R, \frac{\eps}{4}}$
  and therefore is again implemented by a $\sigma$-network
  whose depth and width are $\CalO \big( \ln( R^2 \, \eps^{-1} ) \big)$,
  whose number of weights and neurons are $\CalO \big( \ln^2( R^2 \, \eps^{-1} ) \big)$,
  and whose weights are bounded by $\CalO \big( R^{16} \eps^{-7} \big)$ in absolute value.
\end{proof}

\section{Partition of unity}
\label{sec:PartitionOfUnity}

%%%%%%%%%%%%%%%%%%%%%%%%%%%%%%%%%%%%%%%%%%%%%%%%%%%%%%%%%%%%
%\input{PartitionOfUnity.tex}		%!TEX root=./Draft.tex

Define the functions $\psi^{\Re},\psi^{\Im}:\CC\to\CC$ by
\begin{equation}
  \psi^{\Re}(z)
  := 1-\sigma(z+\tfrac{1}{2})+\sigma(z-\tfrac{1}{2})
  \quad \text{and} \quad
  \psi^{\Im}(z)
  := 1+i\, \sigma(z+\tfrac{1}{2}i)-i\, \sigma(z-\tfrac{1}{2}i)
  .
  \label{eq:PartitionOfUnityBuildingBlocks}
\end{equation}
Note for $x\in\R$ that
\[
  \psi^{\Re}(x)
  =\begin{cases}
    1                 & \text{if } |x| \leq \frac{1}{2}                \\
    \frac{3}{2} - |x| & \text{if } \frac{1}{2} \leq |x| \leq \frac{3}{2} \\
    0                 & \text{if } |x| \geq \frac{3}{2}
  \end{cases}
\]
and for $z \in \CC$ that $\psi^{\Im} (i z) = \psi^{\Re}(z)$,
since $\sigma(iz) = i \, \sigma(z)$.

\begin{figure}[ht]
  \begin{center}
    \includegraphics[width=0.45\textwidth]{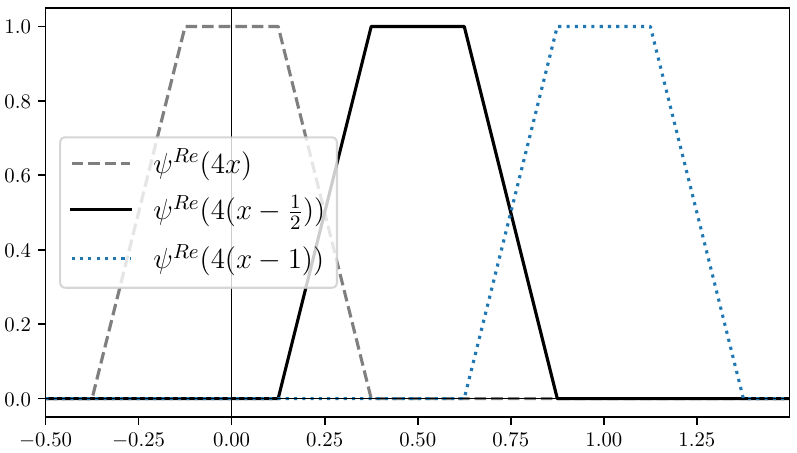}
  \end{center}
  %\captionsetup{width=.8\linewidth,font=small}
  \caption{\label{fig:PartitionOfUnityPlot}
    A plot of the function $\psi^{\Re}(4 \bullet)$ and its shifts,
    showing that they form a partition of unity.}
\end{figure}

Let $N \geq 1$ be a natural number.
For $m \in \{ 0, 1, \dots, 2N \}$ define the functions $\phi_{m,N}^{\Re} : \CC \to \CC$
by $\phi_{m,N}^{\Re} (z) := \psi^{\Re}(4N(z-\frac{m}{2N}))$.
It is not difficult to see that the $\phi_{m,N}^{\Re}$ ($m \in \{ 0,1,\dots,2N \}$)
form a partition of unity on the unit interval $[0,1] \subset \R \subset \CC$,
and that
\[
  \supp \big( \phi_{m,N}^{\Re}|_{\R} \big)
  \subset \big\{
            x \in \R
            \colon
            |x - \tfrac{m}{2 N}|
            \leq \tfrac{3}{8} N^{-1}
          \big\}
  ;
\]
see \Cref{fig:PartitionOfUnityPlot}.
Similarly, defining $\phi_{m,N}^{\Im} (z) := \psi^{\Im}(4N (z - \frac{i m}{2N}))$
for $m \in \{ 0, 1, \dots, 2N \}$, we see that the $\phi_{m,N}^{\Im}$
form a partition of unity on the imaginary unit interval $i \cdot [0,1] \subset \CC$.

\section{Main result}
\label{sec:MainResult}

%%%%%%%%%%%%%%%%%%%%%%%%%%%%%%%%%%%%%%%%%%%%%%%%%%%%%%%%%%%%
%\input{MainResult.tex}		%!TEX root=./Draft.tex

In this section, we prove our main result, \Cref{thm:IntroductionMainTheorem}.
As a preparation for the proof, we collect the following technical lemma,
whose proof is deferred to \Cref{sub:TechnicalMultiplicationResultProof}.

\begin{lemma}\label{lem:LipschitzLemma}
  Let $\Omega \neq \emptyset$ be a set, $M \in \N$, $\eps \in (0, \frac{1}{M+1})$,
  and $0 < \delta \leq \eps^2$.
  Suppose that
  \begin{itemize}[itemsep=0.2cm,topsep=0.2cm]
    \item $\approxMult : \CC^2 \to \CC$ satisfies $|\approxMult (z,w) - z w| \leq \eps$
          for all $|z|, |w| \leq 4$;

    \item $\alpha_1,\dots,\alpha_M : \Omega \to \CC$ satisfy $|\alpha_j (z)| \leq 1$
          for all $z \in \Omega$;

    \item $\beta_1,\dots,\beta_M : \Omega \to \CC$ satisfy $|\alpha_j(z) - \beta_j(z)| \leq \delta$
          for all $z \in \Omega$.
  \end{itemize}
  Define inductively $\gamma_1(z) := \beta_1(z)$
  and $\gamma_{j+1}(z) := \approxMult\bigl(\beta_{j+1}(z), \gamma_j(z)\bigr)$ for $z \in \Omega$.
  Then
  \[
    \Bigl|\gamma_M (z) - \prod_{\ell=1}^M \alpha_\ell (z)\Bigr|
    \leq 3M \, \eps
    \qquad \forall \, z \in \Omega .
  \]
\end{lemma}

\begin{proof}[Proof of \Cref{thm:IntroductionMainTheorem}]
As in \Cref{sub:SmoothnessAssumptions}, we identify the function $g : Q_{\CC^d} \to \CC$
with the pair of functions $g_{\Re},g_{\Im} : [0,1]^{2d} \to \R$
and we will only explicitly show the approximation of
$f := g_{\Re}$, since $g_{\Im}$ can be approximated in exactly the same way.

We roughly follow the structure of the proof of Theorem~1 in \cite{YarotskyReLUBounds}:
In the first step, we approximate $f$ by $f_*$, a sum of Taylor polynomials
subordinate to a partition of unity, constructed with our activation function $\sigma$
in mind; see \Cref{sec:PartitionOfUnity}.
In the second step we approximate $f_*$ by the realization $\widetilde{f}$
of a $\sigma$-network of an appropriate architecture.
An additional complication compared to the real setting considered in \cite{YarotskyReLUBounds}
is that we cannot access the real and imaginary parts of the inputs of $f$ exactly
with a $\sigma$ network, but only approximatively; see \Cref{prop:RealImaginary}.

\medskip{}

\textbf{Step~1.} 
Employing similar notations to \cite{YarotskyReLUBounds},
we will denote ordered pairs (vectors) of coordinates by bold-faced characters.
Given $N \in \N$ (specified precisely in \Cref{eq:NDefinition} below), let us write
\[
  \overline{N}
  := \{0,1,\dots,2N\}^d \times \{0,1,\dots,2N\}^d .
\]
For $\m:=(m_1,m_2,\dots,m_{2d})\in \overline{N}$, we define on $Q_{\CC^d} \cong [0,1]^{2d}$
the function
\[
  \phi_{\m}(\x)
  = \phi_{N,\m}(\x)
  = \prod_{k=1}^d
      \psi^{\Re}\Bigl(4N \cdot \bigl(x_k - \tfrac{m_k}{2N}\bigr)\Bigr)
    \prod_{\ell=d+1}^{2d}
      \psi^{\Im}\Bigl(4N \cdot \bigl(x_\ell \, i - \tfrac{i m_\ell}{2N}\bigr)\Bigr),
\]
where $\x = (x_1,\dots,x_d,x_{d+1},\dots,x_{2d})$ and where $\psi^{\Re},\psi^{\Im}$
are given by \Cref{eq:PartitionOfUnityBuildingBlocks}.

Based on the observations in \Cref{sec:PartitionOfUnity}, we see that
the $\phi_{\m}$ ($\m \in \overline{N}$) form a partition of unity on $[0,1]^{2 d}$ and
satisfy $\supp (\phi_{\m}) \subset S_{\m}$ for the set
\[
  S_{\m}
  := \Big\{
       \x \in \R^{2 d}
       :
       |x_k-\tfrac{m_k}{2N}| < \tfrac{1}{2N}
       \text{ and }
       |x_\ell - \tfrac{m_\ell}{2N}| < \tfrac{1}{2N}
       \text{ for all } 1\leq k\leq d < \ell\leq 2d
     \Big\}.
\]

Now for any $\m\in\overline{N}$, consider the Taylor polynomial
of $f$ at the point $\x=\frac{\m}{2N}$ of degree $n-1$, given by
\[
  P_{\m}(\x)
  = \sum_{\n \in \N_0^{2d}, |\n| < n}
      \bigg[
        \frac{\partial^{\n} f (\frac{\m}{2 N})}{\n !}
        \cdot \bigg(\x - \frac{\m}{2N} \bigg)^{\n}
      \bigg] ,
  \qquad \text{and define} \qquad
  f_*
  := \sum_{\m\in\overline{N}}
       \phi_{\m} P_{\m}
  .
\]

For any $\x \in [0,1]^{2 d}$, we can bound the error by
\begin{align*}
  |f(\x)-f_*(\x)|
  & = \Big|
       \sum_{\m \in \overline{N}}
         \phi_{\m}(\x)
         \bigl(f(\x)-P_{\m}(\x)\bigr)
      \Big|
    \leq \sum_{\m:\x\in S_{\m}}
           \bigl|f(\x) - P_{\m}(\x)\bigr|
           \phi_{\m} (\x) \\
  & \leq \max_{\m:\x\in S_{\m}}
           \bigl| f(\x) - P_{\m}(\x) \bigr|
         \sum_{\m \in \overline{N}} \phi_{\m}(\x)
    \leq \frac{(2d)^n}{n!}
         \bigg(\frac{1}{N}\bigg)^n
         \|f\|_{\mathcal{W}^{n,\infty}} \\
  & \leq \frac{(2d)^n}{n!}
         \bigg(\frac{1}{N}\bigg)^n,
\end{align*}
where, similar to the arguments on Page~108 of \cite{YarotskyReLUBounds},
we used successively the fact that the $\phi_{\m}$ form a partition of unity
and are supported on $S_{\m}$, a standard bound for the error of approximation by the
Taylor polynomial (see e.g.\ the proof of \cite[Lemma~A.8]{PetersenOptimalApproximation}),
and finally that $f$ is in the unit ball of the Sobolev space,
meaning $\| f \|_{\CalW^{n,\infty}} \leq 1$.
Therefore by choosing
\begin{equation}
  N
  := \bigg\lceil
       \bigg(
         \frac{n!\cdot\eps}{2 \cdot (2 d)^n}
       \bigg)^{-1/n}
     \bigg\rceil
  \label{eq:NDefinition}
\end{equation}
(where $\lceil x\rceil$ is the smallest integer bigger or equal to $x$),
we obtain that $\|f - f_*\|_{L^\infty} \leq \frac{\eps}{2}$.

\medskip{}

\textbf{Step~2.}
We approximate $f_\ast$ up to error $\frac{\eps}{2}$ by a $\sigma$-network.
To this end, note that we can rewrite $f_*$ as
\[
  f_*(\x)
  = \sum_{\m \in \overline{N}} \,\,
      \sum_{\n \in \N_0^{2 d} : |\n| < n}
          a_{\m,\n}
          \cdot \phi_{\m}(\x)
          \cdot \bigl(\x - \tfrac{\m}{2N}\bigr)^{\n}
  .
\]
Note that all the coefficients $a_{\m,\n}$ have absolute value at most $1$,
since $\|f\|_{\CalW^{n,\infty}} \leq 1$.
Therefore $f_*$ is a linear combination of no more than $n^{2 d} \, (2N+1)^{2d}$ terms of the form
\[
  f_{\m,\n} (\x)
  := \phi_{\m}(\x)
     \cdot \bigl(\x - \tfrac{\m}{2N}\bigr)^{\n}
  .
\]

Fix $\m \in \overline{N}$ and $\n \in \N_0^{2 d}$ with $|\n| < n$ for the moment.
We want to approximate the function $f_{\m,\n}$ via \Cref{lem:LipschitzLemma}.
Thus, set $S := n^{2 d} \cdot |\overline{N}|$ and $M := 2 d + |\n| < 2d + n$, and
$\widetilde{\eps} := \frac{\eps}{6 (2d + n) \cdot S}$, as well as $\delta := \widetilde{\eps}^2$,
and finally
\(
  \Omega
  := Q_{\CC^d}
  = \bigl\{
       \z \in \CC^d
       \colon
       \Re(z_j), \Im(z_j) \in [0,1]
       \text{ for } j \in \{ 1,\dots,d \}
     \bigr\}
  .
\)

As a first step, we estimate $\widetilde{\eps}$.
Directly from \Cref{eq:NDefinition}, we see
\[
  N
  \leq 1
       + \Big(
           \frac{2 \cdot (2 d)^n}{n! \cdot \eps}
         \Big)^{1/n}
  \leq 1 + 4 d \cdot \eps^{-1}
  \leq 5 d \cdot \eps^{-1} .
\]
Thus, $S = n^{2 d} \, (2N+1)^{2d} \leq (12 d n)^{2 d} \cdot \eps^{-2d}$,
whence ${\widetilde{\eps}^{-1} \leq C_1(d,n) \cdot \eps^{-2d-1} \leq C_1(d,n) \cdot \eps^{-3d}}$.
Therefore,
$\ln(2/\widetilde{\eps}) \leq \ln(2 C_1(d,n)) + 3d \, \ln(1/\eps) \leq C_2(d,n) \cdot \ln(2/\eps)$
for suitable constants $C_1(d,n) \geq 1$ and $C_2(d,n) \geq 1$.

Thus, \Cref{cor:ProductApprox} (applied with $\widetilde{\eps}$ instead of $\eps$) yields a function
$\approxMult : \CC^2 \to \CC$ satisfying $|\approxMult(z,w) - z \, w| \leq \widetilde{\eps}$
for all $z,w \in \CC$ with $|z|, |w| \leq 4$, and such that $\approxMult$ is implemented
by a $\sigma$-network with width and depth bounded by
$C_3 \ln(2/\widetilde{\eps}) \leq C_4 \cdot \ln(2/\eps)$
and at most $C_3 \cdot \ln^2 (2 / \widetilde{\eps}) \leq C_4 \cdot \ln^2 (2/\eps)$ weights,
each bounded in absolute value by $C_3 \cdot \widetilde{\eps}^{-7} \leq C_4 \cdot \eps^{-21 d}$.
Here, $C_3 \geq 1$ is an absolute constant and $C_4 = C_4(d,n) \geq 1$.

Next, note that $\frac{8N}{\widetilde{\eps}^2} \leq C_5 (d,n) \cdot \eps^{-7d}$.
Therefore, \Cref{prop:RealImaginary} shows that there exist functions
$\Re^\natural, \Im^\natural : \CC \to \CC$
with ${|\Re^{\natural}(z) - \Re (z)| \leq \frac{\widetilde{\eps}^2}{8N}}$
and $|\Im^{\natural}(z) - \Im(z)| \leq \frac{\widetilde{\eps}^2}{8N}$
for all $z \in \CC$ with $|z| \leq 4$,
and such that $\Re^{\natural}$ and $\Im^{\natural}$ are implemented by shallow $\sigma$-networks
with $10$ weights of magnitude at most
$C_5 \, (\frac{\widetilde{\eps}^2}{8N})^{-3} \leq C_6 \cdot \eps^{-21 d}$,
for suitable $C_6 = C_6(d,n) \geq 1$.

Finally, to apply \Cref{lem:LipschitzLemma}, writing $\n = (n_1,\dots,n_{2d})$,
we define ${\alpha_k, \beta_k : \CC^d \to \CC}$ for ${1 \leq k \leq 2d+|\n| = M}$ as follows:
\begin{itemize}
  \item For $1\leq k\leq d$, set
        \[
          \alpha_k(\z) := \psi^{\Re}\bigl(4N \Re(z_k) - 2 m_k \bigr)
          \quad \text{ and } \quad
          \beta_k(\z) := \psi^{\Re}\bigl(4N \Re^\natural (z_k) - 2 m_k \bigr);
        \]

  \item For $d+1\leq k\leq 2d$, set
        \[
          \alpha_k(\z) := \psi^{\Im}\bigl( 4N i \Im (z_{k-d}) - 2 i m_k \bigr)
          \quad \text{ and } \quad
          \beta_k(\z) := \psi^{\Im}\bigl(4N i \Im^\natural (z_{k-d}) - 2 i m_k \bigr);
        \]

  \item For $2d+n_1+\dots+n_{\ell-1} < k\leq 2d+n_1+\dots+n_{\ell}\leq 2d+n_1+\dots+n_d$
        ($1 \leq \ell \leq d$), set
        \[
          \alpha_k(\z) := \Re(z_\ell) - \frac{m_\ell}{2N}
          \quad \text{ and } \quad
          \beta_k(\z) := \Re^\natural (z_\ell) - \frac{m_\ell}{2N};
        \]

  \item For $2d+n_1+\dots+n_d  \leq  2d+n_1+\dots+n_{\ell-1}<k\leq 2d+n_1+\dots+n_\ell$
        ($d+1 \leq \ell \leq 2d$), set
        \[
          \alpha_k(\z) := \Im(z_{\ell-d}) - \frac{ m_{\ell} }{2N}
          \quad \text{ and } \quad
          \beta_k(\z) := \Im^\natural (z_{\ell-d}) - \frac{ m_{\ell} }{2N}.
        \]
\end{itemize}
Since $\psi^{\Re},\psi^{\Im}$ are $2$-Lipschitz (this follows from the definition of
$\psi^{\Re},\psi^{\Im}$ and from \Cref{lem:SigmaLipschitz}), we see that
$|\alpha_k(\z) - \beta_k(\z)| \leq \widetilde{\eps}^2 = \delta \leq 1$
for all $\z \in \Omega$ and $1 \leq k \leq M$.
Furthermore, note that indeed $|\alpha_k(\z)| \leq 1$
(and hence $|\beta_k (\z)| \leq 2$) for all $\z \in \Omega$ and $1 \leq k \leq M$.

Overall, we can thus apply \Cref{lem:LipschitzLemma}, which shows for
\[
  \widetilde{f}_{\m,\n}
  := \approxMult
     \big(
       \beta_1,
       \approxMult \big(
         \beta_2,
         \dots,
         \approxMult(\beta_{2d + |\n| - 1}, \beta_{2d + |\n|})
       \big)
     \big)
\]
that
\[
  |f_{\m,\n}(\z) - \widetilde{f}_{\m,\n}(\z)|
  \leq 3 M \, \widetilde{\eps}
  \leq \frac{\eps}{2 S}
  = \frac{\eps}{2} \frac{1}{n^{2 d} |\overline{N}|}
  \qquad \forall \, \z \in \Omega = Q_{\CC^d},
\]
where we identify $\z = (z_1,\CompressedDots,z_d) \in Q_{\CC^d}$ with
$\bigl(\Re(z_1),\CompressedDots,\Re(z_d), \Im(z_1),\CompressedDots,\Im(z_d)\bigr) \in [0,1]^{2 d}$.

Thus, setting
\(
  \widetilde{f}
  := \sum_{\m \in \overline{N}} \,
       \sum_{\n \in \N_0^{2d}, |\n| < n}
         a_{\m,\n} \widetilde{f}_{\m,\n}
\)
and recalling that $|a_{\m,\n}| \leq 1$, we see for any $\z \in \Omega$ that
\[
  |f_\ast (\z) - \widetilde{f}(\z)|
  \leq \sum_{\m \in \overline{N}} \,
         \sum_{\n \in \N_0^{2d}, |\n| < n}
           |a_{\m,\n}| |f_{\m,\n}(\z) - \widetilde{f}_{\m,\n}(\z)|
  \leq |\overline{N}| \cdot n^{2 d} \cdot \frac{\eps}{2} \frac{1}{n^{2 d} |\overline{N}|}
  =    \frac{\eps}{2} ,
\]
and hence $|f(\z) - \widetilde{f}(\z)| \leq \eps$ for all $\z \in Q_{\CC^d}$,
thanks to the bound from Step~1.

\begin{figure}[htb]
\centering
\scalebox{0.65}{%
\begin{tikzpicture}
\begin{scope}%[every node/.style={rectangle,thick,draw}]
\node(L00) at (0,0) {$\z$};

\node(L1-1) at (-5,2) {$\Re^\natural z_1,\dots,\Re^\natural z_d$};
\draw [draw=black] (-5-1.7,2-.35) rectangle ++(3.4,.7);

\node(L11) at (5,2) {$\Im^\natural z_1,\dots,\Im^\natural z_d$};
\draw [draw=black] (5-1.7,2-.35) rectangle ++(3.4,.7);

\node(L2-2) at (-6,5) {$\beta_k(\z),k\in \N_1^d$};
\draw [draw=black] (-6-1.2,5-.35) rectangle ++(2.4,.7);

\node(L2-1) at (-3.3,5) {$\beta_k(\z),k\in \N_{d+1}^{2d}$};
\draw [draw=black] (-3.3-1.3,5-.35) rectangle ++(2.6,.7);

\node(L21) at (-.3,5) {$\beta_k(\z),k\in \N_{2d+1}^{2d+s}$};
\draw [draw=black] (-.3-1.4,5-.35) rectangle ++(2.8,.7);

\node(L22) at (3,5) {$\beta_k(\z),k\in \N_{2d+s+1}^{2d+|n|}$};
\draw [draw=black] (3-1.7,5-.35) rectangle ++(3.4,.7);
%\draw [draw=black] (3-1.7,5-.35) rectangle ++(3.4,.7);
%\node(L2-3) at (-3,6) {$\beta_d(\textbf{z})$};
%\node(L2-2) at (-1.6,6) {$\beta_{d+1}(\textbf{z})$};
%\node(L2-1) at (0.4,6) {$\beta_{2d}(\textbf{z})$};
%\node(L20) at (2,6) {$\beta_{2d+1}(\textbf{z})$};
%\node(L21) at (4,6) {$\beta_{2d+s}(\textbf{z})$};
%\node(L22) at (5.8,6) {$\beta_{2d+s+1}(\textbf{z})$};

% below box checked
\node(L3-2) at (-6,8) {$\beta_k(\z),k\in \N_1^d$};
\draw [draw=black] (-6-1.2,8-.35) rectangle ++(2.4,.7);

\node(L3-1) at (-3.3,8) {$\beta_k(\z),k\in \N_{d+1}^{2d}$};
\draw [draw=black] (-3.3-1.3,8-.35) rectangle ++(2.6,.7);

\node(L31) at (-.3,8) {$\beta_k(\z),k\in \N_{2d+1}^{2d+s}$};
\draw [draw=black] (-.3-1.4,8-.35) rectangle ++(2.8,.7);

\node(L32) at (3,8) {$\beta_k(\z),k\in \N_{2d+s+1}^{2d+|n|-2}$};
\draw [draw=black] (3-1.7,8-.35) rectangle ++(3.4,.7);

\node(L33) at (8,8) {$u(\z):=\widetilde{\times}(\beta_{2d+|n|-1}(\z),\beta_{2d+|n|}(\z))$};
\draw [draw=black] (8-2.9,8-.35) rectangle ++(5.8,.7);

\node(L4-2) at (-6,11) {$\beta_k(\z),k\in \N_1^d$};
\draw [draw=black] (-6-1.2,11-.35) rectangle ++(2.4,.7);

\node(L4-1) at (-3.3,11) {$\beta_k(\z),k\in \N_{d+1}^{2d}$};
\draw [draw=black] (-3.3-1.3,11-.35) rectangle ++(2.6,.7);

\node(L41) at (-.3,11) {$\beta_k(\z),k\in \N_{2d+1}^{2d+s}$};
\draw [draw=black] (-.3-1.4,11-.35) rectangle ++(2.8,.7);

\node(L42) at (3,11) {$\beta_k(\z),k\in \N_{2d+s+1}^{2d+|n|-3}$};
\draw [draw=black] (3-1.7,11-.35) rectangle ++(3.4,.7);

\node(L43) at (8,11) {$\widetilde{\times}(\beta_{2d+|n|-2}(\z),u(\z))$};
\draw [draw=black] (8-2.9,11-.35) rectangle ++(5.8,.7);

\node(L5-2) at (-6,14) {$  \vdots$};
\node(L5-1) at (-3.3,14) {$ \vdots$};
\node(L51) at (-.3,14) {$  \vdots$};
\node(L52) at (3,14) {$  \vdots$};
\node(L53) at (8,14) {$  \vdots$};

\node(L6-2) at (-6,17) {$\beta_1(\z)$};
\draw [draw=black] (-6-.5,17-.35) rectangle ++(1,.7);

\node(L60) at (5,17) {
                      \(
                        \widetilde{\times}(
                          \beta_2(\z),
                          \dots,
                          \widetilde{\times}(
                            \beta_{2d+|\n|-1}(\z),
                            \beta_{2d+|\n|}(\z)
                          )
                        )
                      \)
                     };
\draw [draw=black] (5-3.4,17-.35) rectangle ++(6.8,.7);

\node(L70) at (0,19) {$\widetilde{f}_{\m,\n}(\z)$};
\draw [draw=black] (0-.7,19-.35) rectangle ++(1.4,.7);
\end{scope}

\begin{scope}[every node/.style={fill=white,rectangle},
              every edge/.style={draw=blue, thick}]
  \path [->] (L00) edge node {$\Re^\natural$} (L1-1);
  \path [->] (L00) edge node {$\Im^\natural$} (L11);

  \path [->] (L1-1) edge node {$\psi^{\Re}(4N\square-2m)$} (L2-2);
  \path [->] (L1-1) edge node {$\square-\frac{m}{2N}$} (L21);
  \path [->] (L11) edge node {$\psi^{\Im}(4 i N \square - 2 i m)$} (L2-1);
  \path [->] (L11) edge node {$\square-\frac{m}{2N}$} (L22);

  \path [->] (L2-2) edge node {$\text{Id}_R$} (L3-2);
  \path [->] (L2-1) edge node {$\text{Id}_R$} (L3-1);
  \path [->] (L21) edge node {$\text{Id}_R$} (L31);
  \path [->] (L22) edge node {$\text{Id}_R$} (L32);
  \path [->] (L22) edge node {$\widetilde{\times}$} (L33);

  \path [->] (L3-2) edge node {$\text{Id}_R$} (L4-2);
  \path [->] (L3-1) edge node {$\text{Id}_R$} (L4-1);
  \path [->] (L31) edge node {$\text{Id}_R$} (L41);
  \path [->] (L32) edge node {$\text{Id}_R$} (L42);
  \path [->] (L32) edge node {$\widetilde{\times}$} (L43);
  \path [->] (L33) edge node {$\widetilde{\times}$} (L43);

  \path [->] (L4-2) edge node {$\text{Id}_R$} (L5-2);
  \path [->] (L4-1) edge node {$\text{Id}_R$} (L5-1);
  \path [->] (L41) edge node {$\text{Id}_R$} (L51);
  \path [->] (L42) edge node {$\text{Id}_R$} (L52);
  \path [->] (L42) edge node {$\widetilde{\times}$} (L53);
  \path [->] (L43) edge node {$\widetilde{\times}$} (L53);

  \path [->] (L5-2) edge node {$\text{Id}_R$} (L6-2);
  \path [->] (L5-2) edge node {$\widetilde{\times}$} (L60);
  \path [->] (L53) edge node {$\widetilde{\times}$} (L60);

  \path [->] (L6-2) edge node {$\widetilde{\times}$} (L70);
  \path [->] (L60) edge node {$\widetilde{\times}$} (L70);
\end{scope}
\end{tikzpicture}
}
\caption{Schematic of the architecture of the network implementing $\widetilde{f}_{\m,\n}$.
For brevity, the figure uses the notation $\N_k^\ell:=\{k,k+1,\dots,\ell-1,\ell\}$
and $s := n_1 + \cdots + n_d$ as well as $R := 2$.}
\label{fig:cartoonOfNEtwork}
\end{figure}
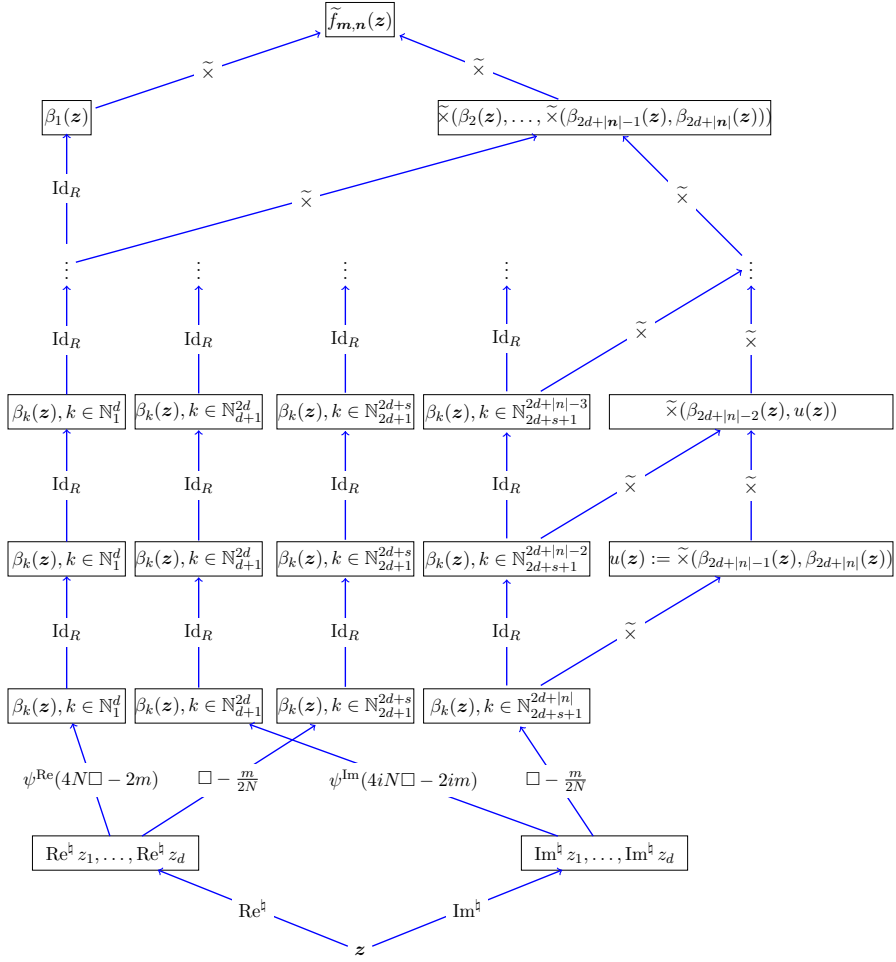

\medskip{}

\textbf{Step~3 (Size of the network):}
Note that $\widetilde{f}_{\m,\n}$ can be expressed as a composition of the networks
$\text{Id}_R$ (with $R = 2$) and $\Re^\natural,\Im^\natural, \psi^{\Re},\psi^{\Im}$,
as well as $\widetilde{\times}$ (see \Cref{fig:cartoonOfNEtwork})
and that the number of such subnetworks that appear in $\widetilde{f}_{\m,\n}$
depends only on the dimension $d$ and the degree of smoothness $n$.

Next, note that with the implied constants (potentially) depending on $d$ and $n$,
the following hold:

\begin{itemize}[leftmargin=0.7cm,itemsep=0.2cm]
  \item $\text{Id}_R$ (with $R = 2$) is implemented by a $\sigma$-network
        with $\CalO(1)$ weights and layers, and all weights bounded by $\CalO(1)$;

  \item $\Re^\natural$ and $\Im^\natural$ are implemented by $\sigma$-networks
        with $\mathcal{O}(1)$ weights and layers
        and all weights bounded by $\mathcal{O}(\eps^{-21d})$;

  \item $\psi^{\Re}$ and $\psi^{\Im}$ are implemented by $\sigma$-networks with
        $\mathcal{O}(1)$ weights and layers and all weights bounded by $\CalO(1)$;
        hence, $\psi^{\Re}(4N \bullet - 2 m)$ and $\psi^{\Im}(4 i N \bullet - 2 i m)$
        (for $m \in \{ 0,1,\dots,2N \}$) are implemented by $\sigma$-networks
        with $\CalO(1)$ weights and layers and all weights bounded by
        $\CalO(N) \subset \CalO(\eps^{-1})$;

  \item $\widetilde{\times}$ is implemented by a $\sigma$-network
        with depth and width bounded by $\mathcal{O}(\ln(2/\eps))$
        and $\mathcal{O}( \ln^2(2/\eps) )$ weights,
        bounded in absolute value by $\mathcal{O}(\eps^{-21d})$.
\end{itemize}

\smallskip{}
\noindent
This implies that $\widetilde{f}_{\m,\n}$ is implemented by a $\sigma$-network
$\Phi^{\eps}_{\m,\n}$ satisfying $W \! (\Phi^{\eps}_{\m,\n}) \!\in\! \CalO(\ln^2(2/\eps))$,
$B(\Phi^{\eps}_{\m,\n}),L(\Phi^{\eps}_{\m,\n}) \in \CalO(\ln(2/\eps))$,
and $\| \Phi^{\eps}_{\m,\n} \| \in \CalO(\eps^{-42 d})$,
where the implied constants (only) depend on $d,n$.
For the full details, we refer to
\Cref{sub:DNNCalculusComposition,sub:DNNCalculusAddition}.

Now, since $\widetilde{f}$ is a linear combination of the no more than $n^{2 d} (2N+1)^{2d}$
functions $\widetilde{f}_{\m,\n}$ with coefficients no larger in absolute value than $1$
and recalling that $N \lesssim_{d,n} \eps^{-1/n}$ (see \Cref{eq:NDefinition}),
it follows that for some $C = C(d,n) > 0$ independent of $\eps$ and $f$,
the function $\widetilde{f}$ is implemented by a $\sigma$-network with
no more than $C \cdot \ln(2/\eps)$ layers
and no more than $C\cdot\eps^{-2d/n} \cdot \ln^2(2/\eps)$ weights,
each bounded in absolute value by $C \cdot \eps^{-44 d}$.

\smallskip{}

Finally,  note that
\(
  \widetilde{f}
  = \sum_{\m \in \overline{N}}
      \sum_{\n \in \N_0^{2 d}, |\n| < n}
        a_{\m,\n} \widetilde{f}_{\m,\n}
\)
where only the coefficients $a_{\m,\n}$ depend on $f$, whereas the functions $\widetilde{f}_{\m,\n}$
are independent of $f$.
This easily implies that one can choose a fixed network architecture $\CalA$ (only depending
on $d,n,\eps$ but independent of $f$) with $L(\CalA) \leq C \cdot \ln(2/\eps)$
and $W(\CalA) \leq C \cdot \eps^{-2d/n} \cdot \ln^2(2/\eps)$ such that $\widetilde{f}$
is implemented by a $\sigma$-network $\Phi_{f}$ of architecture $\CalA$
and with $\| \Phi_f \| \leq C \cdot \eps^{-44 d}$.
\end{proof}

\section{Optimality}%
\label{sec:Optimality}

In this section, we show that the approximation rate obtained in
\Cref{thm:IntroductionMainTheorem} cannot be improved significantly.
Precisely, we prove the following result:

\begin{theorem}\label{thm:Optimality}
  Let $d,n \in \N$, $\theta,\kappa,\gamma \geq 0$, and $C_0 \geq 1$.
  Assume that for every $\eps \in (0,1)$ and every $g \in \CalF_{n,d}$
  there exists a function $g_\eps$ implemented by a $\sigma$-network
  with at most $C_0 \cdot (\ln(2/\eps))^{\kappa}$ layers
  and at most $C_0 \cdot \eps^{-\gamma}$ weights,
  all bounded in absolute value by $C_0 \cdot \eps^{-\theta}$
  satisfying $\| g - g_\eps \|_{L^\infty} \leq \frac{\eps}{2}$.

  Then necessarily $\gamma \geq 2d / n$.
\end{theorem}

The proof idea consists in showing that the set of $\sigma$-networks
of a given complexity satisfies certain \emph{entropy bounds}.
If the approximation rate from \Cref{thm:IntroductionMainTheorem} could be
strictly improved, this would then imply entropy bounds for the set
$\CalF_{n,d}$ that contradict the known asymptotics of the
entropy numbers of $\CalF_{n,d}$ \cite{ClementsEntropies}.

We will derive the entropy bounds for the network sets
as a consequence of the following Lipschitz bound
for the realization map $\Phi \mapsto R_\sigma \Phi$.
Since this bound is quite similar to the one in
\cite[Theorem~2.6]{BlackScholesGeneralizationError}---%
although there only \emph{real-valued} networks with the ReLU activation function
are considered---we postpone the proof to \Cref{sub:RealizationLipschitzBoundProof}.

\begin{lemma}\label{lem:RealizationLipschitzBound}
  Given two networks $\Phi = \big( (A_\ell, b_\ell) \big)_{\ell=1}^L$
  and $\Psi = \big( (B_\ell, c_\ell) \big)_{\ell=1}^L$
  such that for each $\ell \in \{ 1,\dots,L \}$, the matrices
  $A_\ell,B_\ell$ and vectors $b_\ell,c_\ell$ have the same dimension,
  define $\Phi-\Psi$ to be the network $\big( (A_\ell-B_\ell,b_\ell-c_\ell)\big)_{\ell=1}^L$. 

  Let $R,R_0 \geq 1$ and assume $\| \Phi \|, \| \Psi \| \leq R$.
  Then, for every $z \in \CC^{N_0}$ with $\| z \|_{\ell^\infty} \leq R_0$, we have
  \[
    |R_\sigma \Phi (z) - R_\sigma \Psi (z)|
    \leq R_0 \cdot N_1 \cdots N_L \cdot 4^L R^{L-1} \cdot \| \Phi - \Psi \| .
  \]
\end{lemma}

As the final preparation for the proof of \Cref{thm:Optimality},
we recall the notion of covering numbers and a few related facts.
Given a non-empty subset $\emptyset \neq M \subset X$ of a metric space $(X,d)$,
the \emph{covering number} $\Covering(M,\eps) = \Covering_X (M,\eps) \in \N \cup \{ \infty \}$
is the minimal number $n \in \N$ of elements $m_1,\dots,m_n \in M$ satisfying
$M \subset \bigcup_{i=1}^n \ClosedBall_\eps(m_i)$, where
$\ClosedBall_\eps (m) = \{ x \in X \colon d(x,m) \leq \eps \}$
is the closed ball of radius $\eps$ around $m$.

It follows directly from the definitions that if $F : M \subset X \to Y$ is Lipschitz continuous
with $\Lip(F) \leq L$ for some $L > 0$, then
\begin{equation}
  \Covering_{Y} (F(M), \eps)
  \leq \Covering_X (M, \eps/L)
  \label{eq:LipschitzMapCoveringNumberEstimate}
\end{equation}
and that
\begin{equation}
  \Covering \bigl( \textstyle\bigcup_{j=1}^n M_j , \eps\bigr)
  \leq \sum_{j=1}^n
         \Covering(M_j, \eps) .
  \label{eq:CoveringNumberUnionEstimate}
\end{equation}

Using the identification $Q_{\CC^d} \cong [0,1]^{2d}$,
the following bound for the covering numbers of the set $\CalF_{n,d}$
(considered as a subset of $C(Q_{\CC^d})$ with the sup-norm)
is an easy consequence of \cite[Theorem~3 and Theorem on Page~1086]{ClementsEntropies}:

\begin{lemma}\label{lem:FunctionClassCoveringNumbers}
  For $d,n \in \N$ there exists a constant $C_1 = C_1(d,n) > 0$ satisfying
  \[
    \ln \bigl(\Covering_{C(Q_{\CC^d})}(\CalF_{n,d}, \eps)\bigr)
    \geq C_1 \cdot \eps^{-2d/n}
    \qquad \forall \, \eps \in (0,1] .
  \]
\end{lemma}

Furthermore, we will use the following bound for the covering numbers
of subsets of $\R^n$, taken from \cite[Lemma~2.7]{BlackScholesGeneralizationError}:

\begin{lemma}\label{lem:FiniteDimensionCoveringBound}
  Let $n \in \N$, $R \in [1,\infty)$ and $\eps \in (0,e^{-1})$.
  Then, using the $\| \cdot \|_{\ell^\infty}$-norm on $\R^n$, we have
  \[
    \Covering \bigl([-R,R]^n ,\eps\bigr)
    \leq \exp \bigl(n \cdot \ln( \lceil R/\eps \rceil ) \bigr)
    \leq \exp \bigl(2 n \cdot \ln(R/\eps) \bigr)
    =    (R/\eps)^{2 n} .
  \]
\end{lemma}

Using these preparations, we can finally prove \Cref{thm:Optimality}.

\begin{proof}[Proof of \Cref{thm:Optimality}]
  \textbf{Step~1:}
  Given $\eps \in (0,e^{-1})$, set
  $W_\eps := \big\lfloor C_0 \cdot \bigl(\frac{2}{\eps}\bigr)^{\gamma} \big\rfloor$
  and $R_\eps := C_0 \cdot (\frac{2}{\eps})^{\theta}$,
  as well as $L_\eps := \big\lfloor C_0 \cdot \big( \ln(\frac{4}{\eps}) \big)^{\kappa} \big\rfloor$.
  Finally, let
  \[
    \NNSet_\eps
    := \big\{
         R_\sigma \Phi |_{Q_{\CC^d}}
         \,\,\,\colon\,\,\,
         \din (\Phi) = d,
         \dout(\Phi) = 1,
         W(\Phi) \leq W_\eps,
         L(\Phi) \leq L_\eps
         \text{ and }
         \| \Phi \| \leq R_\eps
       \big\} .
  \]
  In this step, we show that
  \begin{equation}
    \ln
    \big(
      \Covering_{C(Q_{\CC^d})}(\NNSet_\eps, \eps/2)
    \big)
    \leq C_1 \cdot \bigl(\ln(2/\eps)\bigr)^{1 + 2 \kappa} \cdot \eps^{-\gamma}
    \label{eq:NetworkSetEntropyBound}
  \end{equation}
  for a suitable constant $C_1 = C_1(d,C_0,\kappa,\gamma,\theta) > 0$ independent of $\eps$.

  To see this, let us write $\FirstN{n} := \{ 1,\dots,n \}$ for $n \in \N$.
  Furthermore, given $L \in \FirstN{L_\eps}$ and $\boldN = (N_1,\dots,N_L) \in \FirstN{W_\eps}^L$
  with $N_L = 1$ and given $\J = (J_1,\dots,J_L)$
  with $J_\ell \subset \FirstN{N_\ell} \times \FirstN{N_{\ell-1}}$ and $|J_\ell| \leq W_\eps$,
  define
  \begin{align*}
    \Lambda_{\boldN,\J} : \quad
    & \prod_{\ell=1}^L
      \Big(
        [-R_\eps,R_\eps]^{J_\ell} \times [-R_\eps,R_\eps]^{N_\ell}
        \times [-R_\eps,R_\eps]^{J_\ell} \times [-R_\eps,R_\eps]^{N_\ell}
      \Big)
      \to C(Q_{\CC^d}), \quad \\
    & \Big(
        \bigl(A^{(\ell)}_{j,k}\bigr)_{(j,k) \in J_\ell},
        b^{(\ell)},
        \bigl(B_{j,k}^{(\ell)}\bigr)_{(j,k) \in J_\ell},
        c^{(\ell)}
      \Big)_{\ell=1}^L
      \mapsto R_\sigma
              \Big(
                \bigl(
                  A^{(\ell)} + i B^{(\ell)},
                  b^{(\ell)} + i c^{(\ell)}
                \bigr)_{\ell=1}^L
              \Big) ,
  \end{align*}
  where $A^{(\ell)}_{j,k} = B^{(\ell)}_{j,k} = 0$ for
  $(j,k) \in (\, \FirstN{N_\ell} \times \FirstN{N_{\ell-1}} \,) \setminus J_\ell$.

  \smallskip{}

  We first claim that $\NNSet_\eps \subset \bigcup_{\boldN,\J} \mathrm{Im}(\Lambda_{\boldN,\J})$,
  where the union is taken over all $\boldN,\J$ as above.
  To see this, note for $f \in \NNSet_\eps$ that $f = R_\sigma \Phi$ for a network
  $\Phi = \smash{\big( (A_\ell,b_\ell) \big)_{\ell=1,\dots,L}}$ that satisfies
  $L \leq L_\eps$, $\sum_{j=1}^L (\| A_j \|_{\ell^0} + \| b_j \|_{\ell^0}) = W(\Phi) \leq W_\eps$,
  and $\| A_\ell \|_{\infty}, \| b_\ell \|_{\infty} \leq R_\eps$.
  Since $\sigma(0) = 0$, it is easy to see simply by dropping ``dead neurons''
  (i.e., neurons that always compute the value $0$, independent of the network input)
  that one can assume $A_\ell \in \CC^{N_\ell \times N_{\ell-1}}$ where
  $N_\ell \leq W_\eps$ for $\ell \in \FirstN{L}$.
  Furthermore, the condition on the number of weights shows that for every
  $\ell \in \FirstN{L}$, one can choose a set
  $J_\ell \subset \FirstN{N_\ell} \times \FirstN{N_{\ell-1}}\strut$
  satisfying $|J_\ell| \leq W_\eps$ and such that $(A_\ell)_{j,k} = 0$ unless $(j,k) \in J_\ell$.
  This easily implies $\strut f = R_\sigma \Phi \in \mathrm{Im}(\Lambda_{\boldN,\J})$.

  Next, note for fixed $L \in \FirstN{L_\eps}$ and $\boldN \in \FirstN{W_\eps}^{L}$ that
  $N_\ell \leq d W_\eps$ (even for $\ell = 0$) and hence
  \begin{equation}
    \begin{split}
      \big|
        \big\{
          J_\ell \subset \FirstN{N_\ell} \times \FirstN{N_{\ell-1}}
          \colon
          |J_\ell| \leq W_\eps
        \big\}
      \big|
      & = \sum_{t=0}^{\min\{ W_\eps, N_\ell N_{\ell-1} \}}
            \binom{N_\ell N_{\ell-1}}{t} \\
      & \overset{(\ast)}{\leq} \bigl(
               e N_\ell N_{\ell-1} \big/ \min\{ W_\eps, N_\ell N_{\ell-1} \}
             \bigr)^{\min\{ W_\eps, N_\ell N_{\ell-1} \}} \\
      & \leq \bigl(d^2 e W_\eps^2 \big/ W_\eps\bigr)^{\min \{ W_\eps, N_\ell N_{\ell-1} \}}
        \leq \bigl(d^2 e W_\eps \bigr)^{W_\eps} .
    \end{split}
    \label{eq:NumberOfSubsetsBound}
  \end{equation}
  Here, the step marked with $(\ast)$ used the elementary bound
  $\sum_{t=0}^m \binom{n}{t} \leq (e n / m)^m$ which is valid for $1 \leq m \leq n$;
  see e.g.~\cite[Exercise~0.0.5]{VershyninHighDimensionalProbablity}.
  
  As the next step, note that \Cref{lem:RealizationLipschitzBound}
  shows that if we equip the domain of $\Lambda_{\boldN,\J}$
  with the $\| \bullet \|_{\ell^\infty}$ norm, then
  $\Lambda_{\boldN,\J}$ is Lipschitz continuous with Lipschitz constant
  \[
    \Lip(\Lambda_{\boldN,\J})
    \leq 2 \cdot (8 W_\eps R_\eps)^{L_\eps}
    \leq \bigl(16 C_0^2 \cdot (2/\eps)^{\gamma+\theta}\bigr)^{C_0 \, (\ln(4/\eps))^{\kappa}}
    =:   \Theta_\eps .
  \]
  Combining this with
  \Cref{eq:LipschitzMapCoveringNumberEstimate,eq:CoveringNumberUnionEstimate,eq:NumberOfSubsetsBound}
  and \Cref{lem:FiniteDimensionCoveringBound}, we therefore see
  \begin{align*}
    \Covering_{C(Q_{\CC^d})}(\NNSet_\eps,\eps/2)
    & \leq \sum_{L,\boldN, \J}
             \Covering_{C(Q_{\CC^d})}\bigl(\mathrm{Im}(\Lambda_{\boldN,\J}), \eps/2\bigr) \\
    & \leq \sum_{L,\boldN, \J}
             \Covering
             \bigl(
               [-R_\eps,R_\eps]^{\sum_{\ell=1}^L (2 |J_\ell| + 2 W_\eps)},
               \tfrac{\eps}{2\Theta_\eps}
             \bigr) \\
    & \leq L_\eps
           \cdot W_\eps^{L_\eps}
           \cdot (d^2 e W_\eps)^{W_\eps}
           \cdot \Bigl(R_\eps \frac{2 \Theta_\eps}{\eps}\Bigr)^{8 L_\eps W_\eps} .
  \end{align*}
  It is straightforward to see that
  $\ln(W_\eps) \leq \ln(d^2 e W_\eps) \lesssim \ln(2/\eps)$
  and $\ln(L_\eps) \lesssim \ln(2/\eps)$, as well as $\ln(R_\eps) \lesssim \ln(2/\eps)$
  and $\ln(\Theta_\eps) \lesssim \bigl(\ln(2/\eps)\bigr)^{\kappa+1}$,
  where the implied constants only depend on $d,C_0,\gamma,\kappa,\theta$.
  In view of these estimates and because of $L_\eps, W_\eps \geq 1$,
  the preceding displayed equation shows
  \[
    \ln \big( \Covering_{C(Q_{\CC^d})}(\NNSet_\eps,\eps/2) \big)
    \lesssim L_\eps \cdot W_\eps \cdot \big( \ln(2/\eps) \big)^{\kappa+1}
    \lesssim \big( \ln(2/\eps) \big)^{1 + 2 \kappa} \cdot \eps^{-\gamma} ,
  \]
  proving \Cref{eq:NetworkSetEntropyBound}.

  \medskip{}

  \textbf{Step~2 (Completing the proof):}
  Let $\eps \in (0,e^{-1})$.
  \Cref{eq:NetworkSetEntropyBound} implies that there exists a constant $M_\eps \in \N$
  with $M_\eps \leq \exp \bigl(C_1 \cdot (\ln(2/\eps))^{1+2\kappa} \cdot \eps^{-\gamma}\bigr)$
  and $\NNSet_\eps \subset \bigcup_{m=1}^{M_\eps} \ClosedBall_{\eps/2}(g^{(\eps)}_m)$
  for suitable $g^{(\eps)}_m \in \NNSet_\eps$, where 
  $\ClosedBall_\delta (g) := \{ f \in C(Q_{\CC^d}) \colon \| f - g \|_{L^\infty} \leq \delta \}$
  is the closed ball of radius $\delta$ around $g$.

  Now, for each $m \in \FirstN{M_\eps}$ choose
  $h_m^{(\eps)} \in \CalF_{n,d} \cap \ClosedBall_{\eps}(g_m^{(\eps)})$
  if this intersection is non-empty, and $h_m^{(\eps)} = 0$ otherwise.
  By assumption of the theorem, for each $f \in \CalF_{n,d}$
  there exists $g \in \NNSet_\eps$ satisfying $\| f - g \|_{L^\infty} \leq \eps/2$.
  Then,
  \(
    \| f - g_m^{(\eps)} \|_{L^\infty}
    \leq \frac{\eps}{2} + \| g - g_m^{(\eps)} \|_{L^\infty}
    \leq \eps
  \)
  for a suitable $m \in \FirstN{M_\eps}$, and hence 
  \(
    \| f - h_m^{(\eps)} \|_{L^\infty}
    \leq \| f - g_m^{(\eps)} \|_{L^\infty}
         + \| g_m^{(\eps)} - h_m^{(\eps)} \|_{L^\infty}
    \leq 2 \eps.
  \)

  Overall, this shows $\CalF_{n,d} \subset \bigcup_{m=1}^{M_\eps} \ClosedBall(h_m^{(\eps)}, 2\eps)$
  and hence
  \[
    \ln \big( \Covering(\CalF_{n,d},2\eps) \big)
    \leq \ln(M_\eps)
    \leq C_1 \cdot (\ln(2/\eps))^{1+2\kappa} \cdot \eps^{-\gamma}
    \qquad \forall \, \eps \in (0,e^{-1}).
  \]
  In view of \Cref{lem:FunctionClassCoveringNumbers}, 
  this is only possible if $\gamma \geq 2d/n$, which is what we wanted to show.
\end{proof}

\section{Conclusion}

In the present paper, we studied the problem of approximating functions of regularity
$C^n$ defined on $\CC^d$ using feed-forward complex-valued neural networks (CVNNs)
with modReLU activation function.
We showed that (ignoring logarithmic factors) a suitably constructed
modReLU CVNN with $\mathcal{O}(\varepsilon^{-2d/n})$ parameters (weights)
can achieve uniform approximation error $\eps$.
Moreover, we showed that this rate is near-optimal.
This is as expected, since comparable real-valued neural networks
obtain the same rates, cf.\ \cite{YarotskyReLUBounds} (identifying $\mathbb C \simeq \mathbb R^2$).

Since it is known that ReLU neural networks achieve \emph{optimal} approximation rates for $C^n$ functions
(see \cite{YarotskyReLUBounds,PetersenOptimalApproximation,BoelcskeiSparseNN}), it cannot be expected that
(modReLU) CVNNs \emph{strictly improve} on ReLU networks, even in the complex setting.
Rather, since CVNNs have been empirically observed to outperform real-valued neural networks
in many applications involving complex-valued inputs \cite{LustigComplexNNForMRI,BengioUnitaryEvolutionRNN},
our goal is to initiate the study of the expressivity of CVNNs;
furthermore, our goal was to rigorously prove that modReLU CVNNs can \emph{match}
the approximation capabilities of ReLU neural networks.
Our results confirm that this is indeed the case.

The essential properties of the modReLU on which our proof relies are the following:
\begin{itemize}
  \item modReLU CVNNs of a \emph{constant size} can approximate the function $z \mapsto \Re(z)$
        (and hence also the function $z \mapsto \Im(z)$) arbitrarily well; see \Cref{prop:RealImaginary};

  \item modReLU CVNNs of a \emph{constant size} can approximate the ``complexified'' ReLU function
        $z \mapsto \varrho( \Re(z) ) = \max \{ 0, \Re(z) \}$ arbitrarily well;
        see \Cref{prop:ReLUreal}; and

  \item the modReLU is Lipschitz continuous.
\end{itemize}
Once these properties are known for a given activation function,
the arguments used to prove our main theorem
(which build upon the ideas in \cite{YarotskyReLUBounds}),
can be used to prove an analogous approximation result for that activation function.
The Lipschitz continuity is used to control the propagation of errors among the layers of the
network; it can probably be replaced by Hölder continuity and possibly even by uniform continuity.
The main technical contribution of the paper is thus to verify that the above properties
are satisfied for the modReLU and to show that these properties imply the main approximation result.

\appendix

\section{Postponed technical proofs}%
\label{sec:PostponedProofs}

\subsection{Proof of Lemma~\ref{lem:RealImaginary}}%
\label{sub:RealImaginaryPartProof}

\begin{proof}
  Set $w :=  h z+\frac 1 h$.
  If $\Im (z) = 0$, then $w \in (0,\infty)$ and hence ${\sgn (w) - 1 = 0 = \Im (z)}$,
  so that the first part of \Cref{eq:RealImagApprox} is true.
  Hence, we can assume in what follows that $\Im (z) \neq 0$.
  Now, note by choice of $h$ that $0 < h \leq \frac{1}{2}$ and $h \, |z| \leq \frac{1}{2}$,
  which shows that $|1 + z h^2| \geq 1 - h \cdot h \, |z| \geq \frac{3}{4}$ and therefore also
  $|w| = \frac{1}{h} |1 + h^2 z| \geq \frac{3}{4 h} > 0$.

  As a consequence, we obtain the estimate
  \begin{align*}
    \Big|
      \frac{1}{h^2}
      \frac{1}{|w|}
      - \frac{1}{h}
    \Big|
    & = \frac{1}{h} \cdot
        \Big|
          \frac{1}{h} \frac{1}{|h^{-1} + h z|} - 1
        \Big|
      = \frac{1}{h} \cdot
        \Big|
          \frac{1}{|1 + h^2 z|} - 1
        \Big| \\
    & = \frac{1}{h} \cdot
        \frac{\big|1 - |1 + h^2 z| \big|}{|1 + h^2 z|}
      \leq \frac{1}{h} \frac{h^2 |z|}{3/4}
      = \frac{4}{3} h \, |z|
      \leq \frac{2}{3}
      \leq 1 .
  \end{align*}
  As $\Im (w) = h \, \Im(z) \neq 0$, this implies
  \begin{align*}
    \Big|
      \frac{1}{h^2} \Im (w / |w|)
      - \Im (z)
    \Big|
    & = h
        \cdot |\Im (z)|
        \cdot \Big|
                \frac{1}{h^2} \frac{\Im(w/|w|)}{\Im (w)}
                - \frac{1}{h}
              \Big| \\
    & = h
        \cdot |\Im(z)|
        \cdot \Big|
                \frac{1}{h^2} \frac{1}{|w|}
                - \frac{1}{h}
              \Big|
      \leq h \, |z| .
  \end{align*}

  Next, note
  \(
    \Re(w)
    = \frac{1}{h} \cdot \bigl(1 + h^2 \Re (z)\bigr)
    \geq \frac{1}{h} \cdot \bigl( 1 - h \cdot h \, |z|\bigr)
    \geq \frac{1}{h} \cdot \bigl(1 - \tfrac{1}{2} \cdot \tfrac{1}{2}\bigr)
    =    \frac{3}{4 h}
    >    0 .
  \)
  Hence, $\Re (w) + |w| \geq 2 \Re (w) \geq \frac{3}{2 h}$.
  Since also $|\Im(w)| = h \, |\Im(z)| \leq h \, |z|$, we thus see
  \[
    \big| \Re(w) - |w| \big|
      = \frac{\big| (\Re(w) - |w|) (\Re(w) + |w|) \big|}{\big|\Re(w) + |w| \big|}
      = \frac{\big| (\Re(w))^2 - |w|^2 \big|}{\big|\Re(w) + |w| \big|}
      \leq \frac{|\Im(w)|^2}{\frac{3}{2 h}}
  \]
  and hence $\big| \Re(w) - |w| \big| \leq \frac{2}{3} h^3 \, |z|^2$.
  Together with the estimate $|w| \geq \frac{3}{4 h}$
  from the beginning of the proof, we get
  \[
    \frac{1}{h^2}
    \Big|
      \Re (w / |w|) - 1
    \Big|
    = \frac{1}{h^2}
      \bigg|
        \frac{\Re(w) - |w|}{|w|}
      \bigg|
    \leq \frac{1}{h^2}
         \frac{\frac{2}{3} h^3 |z|^2}{\frac{3}{4 h}}
    =    \frac{4}{3}
         \frac{2}{3}
         h^2 |z|^2
    \leq h^2 |z|^2 .
  \]
  Combining everything, we arrive at
  \begin{align*}
    \Big|
        \Im(z)
        - \frac{-i}{h^2}
          \cdot \Big(
                  \sgn( h z + \tfrac{1}{h} )   - 1
                \Big)
      \Big|
    & = \Big|
          \frac{-i}{h^2}
          \cdot \Big(
                  \frac{w}{|w|} - 1
                \Big)
          - \Im(z)
        \Big| \\
    & \leq \Big|
             \frac{-i}{h^2} \cdot i \cdot \Im\Bigl(\frac{w}{|w|}\Bigr)
             - \Im (z)
           \Big|
           + \Big|
               \frac{-i}{h^2} \cdot \Big( \Re\Big( \frac{w}{|w|} \Big) - 1 \Big)
             \Big| \\
    & \leq \Big|
             \frac{1}{h^2} \Im(w / |w|) - \Im (z)
           \Big|
           + \frac{1}{h^2} \big| \Re(w / |w|) - 1 \big|
           \\
    & \leq h \, |z| + h^2 |z|^2
      \leq 2 h \, |z| ,
  \end{align*}
  proving the first estimate in \Cref{eq:RealImagApprox}.
  To prove the second estimate in \Cref{eq:RealImagApprox},
  simply note that $\Re(z) = \Im (i z)$ and $\sgn(i w) = i \sgn(w)$;
  hence, we get as claimed that
  \begin{align*}
    2 h \, |z|
    & = 2 h \, |i z|
      \geq \Big|
             \Im (i z) - \frac{-i}{h^2} \cdot \big(  \sgn(h i z +  \tfrac{1}{h} )   - 1 \big)
           \Big| \\
    & =    \Big|
             \Re(z) - \frac{-i}{h^2} \cdot \big( i \sgn(h z - \tfrac{i}{h}) + i \, i \big)
           \Big|
      =    \Big|
             \Re (z)
             - \frac{1}{h^2} \cdot \big( \sgn(h z - \tfrac{i}{h}) + i \big)
           \Big| .
  \end{align*}
\end{proof}

\subsection{Composition of neural networks}%
\label{sub:DNNCalculusComposition}

The composition of several neural networks is clearly again represented by a neural network.
In this appendix we make this statement more precise, showing how the size
of the resulting network is related to the size of the ``input'' networks.
We note that the bounds for modReLU networks that we derive here are slightly
worse than those derived for ReLU networks in \cite[Section~2]{\PetersenRef},
owing to the fact that one can easily implement the identity function using the ReLU
while this seems not to be possible (on all of $\CC$) using the modReLU.

But first, we need some additional notation:
Given a network $\Phi = \big( (A_1,b_1),\dots,(A_L,b_L) \big)$, let us write
$W_{\inn}(\Phi) := \| A_1 \|_{\ell^0} + \| b_1 \|_{\ell^0}$
and $W_{\out}(\Phi) := \| A_L \|_{\ell^0} + \| b_L \|_{\ell^0}$
and furthermore $\| \Phi \|_{\inn} := \max \{ \| A_1 \|_\infty, \| b_1 \|_\infty \}$
and $\| \Phi \|_{\out} := \max \{ \| A_L \|_\infty, \| b_L \|_\infty \}$.
Now, assuming that $L \geq 2$ and given a further network
$\Psi = \big( (B_1,c_1),\dots,(B_M,c_M) \big)$ with $\dout(\Psi) = \din(\Phi)$ and $M \geq 2$,
define the composition of $\Phi,\Psi$ as
\begin{equation}
  \Phi \compose \Psi
  := \big(
       (B_1,c_1),
       \dots,
       (B_{M-1},c_{M-1}),
       (A_1 B_M, b_1 + A_1 c_M),
       (A_2,b_2),
       \dots,
       (A_L,b_L)
     \big) .
  \label{eq:NNCompositionDefinition}
\end{equation}
It is straightforward to verify
$R_\sigma (\Phi \compose \Psi) = R_\sigma \Phi \circ R_\sigma \Psi$
and $L(\Phi \compose \Psi) = L(\Phi) + L(\Psi) - 1$
and $B(\Phi \compose \Psi) \leq \max \{ B(\Phi), B(\Psi) \}$.
The next lemma provides further bounds on the size of $\Phi \compose \Psi$.

\begin{lemma}\label{lem:CompositionBound}
  Let $\Phi^{(1)},\dots,\Phi^{(K)}$ be neural networks
  of depth $L(\Phi^{(i)}) \geq 2$ for $i \in \{ 1,\dots,K \}$
  and satisfying $\dout(\Phi^{(i)}) = \din(\Phi^{(i+1)})$
  for $i \in \{ 1,\dots,K-1\}$.
  Then the following hold:
  \begin{enumerate}[leftmargin=0.5cm,topsep=0.2cm]
    \item If $W_{\inn}(\Phi^{(i)}),W_{\out}(\Phi^{(i)}) \leq C$
          for all $i \in \{ 1,\dots,K \}$ and some $C > 0$, then
          \begin{equation}
            W\bigl(\Phi^{(K)} \compose \cdots \compose \Phi^{(1)}\bigr)
            \leq C^2 \cdot (K-1) + \sum_{i=1}^K W(\Phi^{(i)}).
            \label{eq:CompositionWeightBound}
          \end{equation}

    \item If $\| \Phi^{(i)} \| \leq C$ for all $i \in \{ 1,\dots,K \}$ and some $C \geq 1$
          and $d_{\out}(\Phi^{(i)}) \leq D$ for all $i \in \{ 1,\dots,K \}$, then
          \begin{equation}
            \big\| \Phi^{(K)} \compose \cdots \compose \Phi^{(1)} \big\|
            \leq 2 D \cdot C^2 .
            \label{eq:CompositionNormBound}
          \end{equation}

    \item \(
            L(\Phi^{(K)} \compose \cdots \compose \Phi^{(1)})
            \!=\! \big[
                \sum_{i=1}^K
                  L(\Phi^{(i)})
              \big]
              - (K-1)
          \)
          and
          \(
            R_\sigma (\Phi^{(K)} \compose \cdots \compose \Phi^{(1)})
            \!=\! R_\sigma \Phi^{(K)} \circ \cdots \circ R_\sigma \Phi^{(1)} .
          \)
  \end{enumerate}
\end{lemma}

\begin{remark}\label{rem:NNComposition}
  In particular, \Cref{eq:CompositionWeightBound} shows that if $B,W \geq 1$
  and $W(\Phi^{(i)}) \leq W$ as well as $B(\Phi^{(i)}) \leq B$ for all $i \in \{ 1,\dots,K \}$,
  then
  \(
    W(\Phi^{(K)} \compose \cdots \compose \Phi^{(1)})
    \leq K \cdot (B^2 + W) .
  \)
\end{remark}

\begin{proof}
  Before we prove the general case, we analyze the composition of two networks as in
  \Cref{eq:NNCompositionDefinition}.
  First, note for $A \in \CC^{N \times K}$ and $B \in \CC^{K \times P}$ that
  \[
    \| A B \|_{\ell^0}
    = \sum_{i,j}
        \Indicator_{(A B)_{i,j} \neq 0}
    \leq \sum_{i,j,\ell}
           \Indicator_{A_{i,\ell} \neq 0} \Indicator_{B_{\ell,j} \neq 0}
    \leq \sum_{i,\ell}
         \Big(
           \Indicator_{A_{i,\ell} \neq 0}
           \max_t
             \sum_{j}
               \Indicator_{B_{t,j} \neq 0}
         \Big)
    \leq \| A \|_{\ell^0} \| B \|_{\ell^0} .
  \]
  A similar (but easier) calculation shows that $\| A v \|_{\ell^0} \leq \| A \|_{\ell^0}$
  for $v \in \CC^{K}$.
  Based on these estimates, we see (in the notation of \Cref{eq:NNCompositionDefinition}) that
  \[
    \| b_1 + A_1 c_M \|_{\ell^0} + \| A_1 B_M \|_{\ell^0}
    \leq \| b_1 \|_{\ell^0}
         + \| A_1 \|_{\ell^0}
         + \| A_1 \|_{\ell^0} \| B_M \|_{\ell^0}
    \leq W_{\inn}(\Phi) \cdot (1 + W_{\out}(\Psi)) .
  \]
  Directly from the definition of $\Phi \compose \Psi$, we thus see
  \begin{equation}
    \begin{split}
      W(\Phi \compose \Psi) 
      & \leq W(\Psi) - W_{\out}(\Psi)
             + W(\Phi) - W_{\inn} (\Phi)
             + W_{\inn}(\Phi) (1 + W_{\out}(\Psi)) \\
      & \leq W(\Psi) + W(\Phi) + W_{\inn}(\Phi) W_{\out}(\Psi) .
    \end{split}
    \label{eq:CompositionWeightEstimate}
  \end{equation}

  Next, given $A \in \CC^{N \times K}$ and $B \in \CC^{K \times P}$ it is easy to see
  $| (A B)_{i,j} | \leq K \cdot \| A \|_\infty \| B \|_\infty$.
  Based on this, we see in the notation of \Cref{eq:NNCompositionDefinition}
  that $\| A_1 B_M \|_{\infty} \leq \dout(\Psi) \, \| A_1 \|_\infty \| B_M \|_\infty$
  and
  \(
    \| b_1 + A_1 c_M \|_\infty
    \leq \| b_1 \|_{\infty} + \dout(\Psi) \| A_1 \|_{\infty} \| c_M \|_{\infty}
    \leq \| \Phi \|_{\inn} \cdot ( 1 + \dout(\Psi) \| \Psi \|_{\out} ) .
  \)
  Thus, we see directly from the definition of $\Phi \compose \Psi$ that
  \begin{equation}
    \| \Phi \compose \Psi \|
    \leq \max
         \big\{
           \| \Phi \|, \,\,
           \| \Psi \|, \,\,
           \| \Phi \|_{\inn}
           \cdot (
                   1
                   + \dout(\Psi) \, \| \Psi \|_{\out}
                 )
         \big\} .
    \label{eq:CompositionNormEstimate}
  \end{equation}

  Now, we prove \Cref{eq:CompositionWeightBound,eq:CompositionNormBound} by induction on $K \in \N$.
  For $K = 1$ the claim is trivial.
  Next, assume that the claim holds for some $K \in \N$ and set
  $\Psi := \Phi^{(K)} \compose \cdots \compose \Phi^{(1)}$.

  For proving \Cref{eq:CompositionWeightBound}, note
  $W_{\inn}(\Phi^{(K+1)}) \leq C$ and $W_{\out}(\Psi) = W_{\out}(\Phi^{(K)}) \leq C$.
  Therefore, combining \Cref{eq:CompositionWeightEstimate} with the inductive assumption, we see
  \begin{align*}
    W(\Phi^{(K+1)} \compose \Psi)
    & \leq W(\Phi^{(K+1)})
           + W(\Psi)
           + W_{\inn}(\Phi^{(K+1)}) W_{\out}(\Psi) \\
    & \leq W(\Phi^{(K+1)})
           + C^2 \cdot (K-1)
           + \sum_{i=1}^K
               W(\Phi^{(i)})
           + C \cdot C \\
    & =    C^2 \cdot ( (K+1) - 1) + \sum_{i=1}^{K+1} W(\Phi^{(i)}) ,
  \end{align*}
  completing the induction for \Cref{eq:CompositionWeightBound}.

  To prove \Cref{eq:CompositionNormBound}, note
  $\| \Psi \|_{\out} = \| \Phi^{(K)} \|_{\out} \leq C$
  and $\dout(\Psi) \!= \dout(\Phi^{(K)}) \leq D$
  and use \Cref{eq:CompositionNormEstimate} and the inductive assumption to obtain
  \[
    \| \Phi^{(K+1)} \compose \Psi \|
    \leq \max
         \big\{
           \| \Phi^{(K+1)} \|, \,\,
           \| \Psi \|, \,\,
           \| \Phi^{(K+1)} \|_{\inn}
           \cdot (
                   1
                   + \dout(\Psi) \, \| \Psi \|_{\out}
                 )
         \big\}
    \leq 2 D \cdot C^2 ,
  \]
  completing the induction for \Cref{eq:CompositionNormBound}.

  The last part of the lemma follows by induction
  after noting that $L(\Phi \compose \Psi) = L(\Phi) + L(\Psi) - 1$
  and $R_\sigma (\Phi \compose \Psi) = R_\sigma \Phi \circ R_\sigma \Psi$.
\end{proof}

\subsection{Linear combinations of neural networks}%
\label{sub:DNNCalculusAddition}

In this appendix we show that the linear combinations of neural networks \emph{of a common depth}
can again be implemented as a neural network.
Indeed, let $d,K \in \N$, and for each $j \in \{ 1,\dots,K \}$ let $a_j \in \CC$ and
let $\Phi^{(j)} = \big( (A_1^{(j)},b_1^{(j)}),\dots,(A_L^{(j)}, b_L^{(j)}) \big)$
be a neural network with $\din(\Phi^{(j)}) = d$ and $\dout(\Phi^{(j)}) = 1$
and of common depth $L(\Phi^{(j)}) = L$.
Define $\Psi := \big( (A_1,b_1),\dots,(A_L,b_L) \big)$, where
\[
  A_1 := \begin{pmatrix}
           A_1^{(1)} \\
           \vdots \\
           A_1^{(K)}
         \end{pmatrix} ,
  \qquad
  A_L := \big( a_1 A_L^{(1)} \,\big|\, \cdots \,\big|\, a_K A_L^{(K)} \big) ,
  \qquad \text{and} \qquad
  b_\ell := \begin{pmatrix}
              b_\ell^{(1)} \\
              \vdots \\
              b_\ell^{(K)}
            \end{pmatrix}
\]
for $\ell \in \{ 1,\dots,L-1 \}$, as well as $A_\ell := \mathrm{diag}(A_\ell^{(1)},\dots,A_\ell^{(K)})$
for $\ell \in \{ 2,\dots,L-1 \}$ and $b_L := \sum_{j=1}^K a_j \, b_L^{(j)}$.
It is easy to verify that
\begin{equation}
  \begin{split}
    & R_\sigma \Psi = \sum_{j=1}^K a_j \, R_\sigma \Phi^{(j)},
    \qquad
    L(\Psi) = L,
    \qquad
    W(\Psi) \leq \sum_{j=1}^K W(\Phi^{(j)}), \\
    & B(\Psi) \leq \sum_{j=1}^K B(\Phi^{(j)}),
    \qquad \text{and} \qquad
    \| \Psi \| \leq \sum_{j=1}^K (1 + |a_j|) \big\| \Phi^{(j)} \big\| .
  \end{split}
  \label{eq:LinearCombinationComplexityBound}
\end{equation}
Indeed, all except the first and final of these statements follow directly from the definitions.
To verify the final statement, note by definition of $\| \bullet \|_{\infty}$ that
\begin{align*}
  & \| A_1 \|_{\infty}
    = \max_{j \in \{ 1,\dots,K \}}
        \big\| A_1^{(j)} \big\|_{\infty}
    \leq \max_{j \in \{ 1,\dots,K \}}
           \big\| \Phi^{(j)} \big\| \\
  \text{and} \quad
  & \| b_\ell \|_{\infty}
    = \max_{j \in \{ 1,\dots,K \}}
        \big\| b_\ell^{(j)} \big\|_{\infty}
    \leq \max_{j \in \{ 1,\dots,K \}}
           \big\| \Phi^{(j)} \big\|
    \quad \text{for} \quad
    \ell \in \{ 1,\dots,L-1 \}, \\
  \text{as well as} \quad
  & \| A_L \|_{\infty}
    \leq \max_{j \in \{ 1,\dots,K \}}
           |a_j| \, \big\| A_L^{(j)} \big\|_{\infty}
    \leq \max_{j \in \{ 1,\dots,K \}}
           |a_j| \, \big\| \Phi^{(j)} \big\| \\
    \text{and} \quad
  & \| b_L \|_{\infty}
    \leq \sum_{j=1}^K
           |a_j| \, \| b_L^{(j)} \|
    \leq \sum_{j=1}^K
           |a_j| \, \big\| \Phi^{(j)} \big\|
  ,
\end{align*}
which implies as claimed that $\| \Psi \| \leq \sum_{j=1}^K (1 + |a_j|) \, \| \Phi^{(j)} \|$.

Finally, to verify the first statement, an induction with respect to $\ell \in \{ 1,\dots,L-1 \}$
shows that if we set $T_\ell^{(j)} := A_\ell^{(j)} (\bullet) + b_\ell^{(j)}$
and $T_{\ell} := A_\ell (\bullet) + b_\ell$ and finally
$F_\ell^{(j)} := (\sigma \circ T_\ell^{(j)}) \circ \cdots \circ (\sigma \circ T_1^{(j)})$
and $F_\ell := (\sigma \circ T_\ell) \circ \cdots \circ (\sigma \circ T_1)$, then
$F_\ell (z) = \bigl(F_\ell^{(1)}(z), \dots, F_\ell^{(K)}(z)\bigr)$ for $z \in \CC^d$
and $\ell \in \{ 1,\dots,L-1 \}$.
Based on this, the first statement in \Cref{eq:LinearCombinationComplexityBound} follows from
the definition of the realization map $R_\sigma$ (see \Cref{sub:CVNNs}).

\subsection{Proof of Lemma~\ref{lem:LipschitzLemma}}%
\label{sub:TechnicalMultiplicationResultProof}

\begin{proof}
  Define $\theta_j (z) := \prod_{\ell=1}^j \alpha_\ell (z)$
  and $\kappa_j := \eps \sum_{\ell=1}^j (1 + \eps)^\ell$.
  We will show inductively that $|\gamma_j (z) - \theta_j(z)| \leq \kappa_j$.
  This will imply the claim by taking $j = M$, since we have
  \[(1+\eps)^\ell \leq (1+\eps)^{M+1} \leq \bigl(1 + \tfrac{1}{M+1}\bigr)^{M+1} \leq e \leq 3\]
  and hence $\kappa_j \leq \kappa_M \leq 3M \, \eps$.

  The case $j = 1$ is trivial, since
  \(
    |\gamma_1 (z) - \theta_1 (z)|
    = |\beta_1(z) - \alpha_1(z)|
    \leq \delta
    \leq \eps^2
    \leq \eps
    \leq \kappa_1
    .
  \)
  For the induction step, first note that
  \begin{align*}
    \kappa_j
    \leq \kappa_M
    & = \eps (1 + \eps) \sum_{\ell=0}^{M-1} (1 + \eps)^\ell
      = \eps (1 + \eps) \frac{(1+\eps)^M - 1}{(1+\eps) - 1} \\
    & \leq (1 + \eps)^{M+1}
      \leq \Bigl(1 + \tfrac{1}{M+1}\Bigr)^{M+1}
      \leq e
      \leq 3
  \end{align*}
  and hence $|\gamma_j(z)| \leq |\theta_j(z)| + \kappa_j \leq 4$, since $|\alpha_\ell (z)| \leq 1$
  for all $\ell$, and thus $|\theta_j(z)| \leq 1$.
  Since also $|\beta_{j+1}(z)| \leq \delta + |\alpha_{j+1}(z)| \leq 1 + \delta \leq 4$,
  we see by the properties of $\approxMult$ for any $z \in \Omega$ that
  \begin{align*}
    \big| \gamma_{j+1}(z) - \theta_{j+1}(z) \big|
    & \leq \Bigl|
             \approxMult\bigl(\beta_{j+1}(z), \gamma_j(z)\bigr) - \beta_{j+1}(z) \gamma_j(z)
           \Bigr| \\
    & \quad + \bigl|\beta_{j+1}(z) \gamma_j(z) - \beta_{j+1}(z) \theta_j (z)\bigr| \\
    & \quad + \bigl|\bigl(\beta_{j+1} (z) - \alpha_{j+1}(z)\bigr) \theta_j (z)\bigr| \\
    & \leq \eps + |\beta_{j+1}(z)| \cdot \kappa_j + \delta \cdot |\theta_j(z)| \\
    & \leq \eps + \delta + (1 + \delta) \kappa_j
      \leq \eps (1 + \eps) + (1 + \eps) \kappa_j ,
  \end{align*}
  where the last step used that $\delta \leq \eps^2 \leq \eps$.
  Finally, note by choice of $\kappa_j$ that
  \[
    \eps (1 + \eps) + (1 + \eps) \kappa_j
    = \eps (1 + \eps) + \eps \sum_{\ell=2}^{j+1} (1 + \eps)^\ell
    = \eps \sum_{\ell=1}^{j+1} (1 + \eps)^\ell
    = \kappa_{j+1} .
  \]
  This completes the induction and thus the proof.
\end{proof}

\subsection{Proof of Lemma~\ref{lem:RealizationLipschitzBound}}%
\label{sub:RealizationLipschitzBoundProof}

\begin{proof}
  Set $\sigma_\ell := \sigma$ for $\ell \in \{ 1,\dots,L-1 \}$
  and $\sigma_L := \mathrm{id}_{\CC}$.
  It is easy to see in each case that $\sigma_\ell (0) = 0$;
  furthermore, \Cref{lem:SigmaLipschitz} implies that each $\sigma_\ell$
  is $1$-Lipschitz.
  Now, inductively define $w_0 := v_0 := z$ as well as
  $w_{\ell+1} := \sigma_{\ell+1}(A_{\ell+1} w_\ell + b_{\ell+1})$
  and $v_{\ell+1} := \sigma_{\ell+1}(B_{\ell+1} v_\ell + c_{\ell+1})$
  for $\ell \in \{ 0,\dots,L-1 \}$.
  We then have $R_\sigma \Phi(z) = w_L$ and $R_\sigma \Psi(z) = v_L$.
  We will show inductively that
  $\| v_\ell \|_{\ell^\infty} \leq R_0 \cdot (2 R)^\ell \cdot N_1 \cdots N_{\ell-1}$
  and
  \(
    \| v_\ell - w_\ell \|_{\ell^\infty}
    \leq R_0 \cdot N_1 \cdots N_{\ell-1} \cdot 4^\ell R^{\ell-1} \cdot \| \Phi - \Psi \|
    ,
  \)
  which then implies the claim of the lemma.

  For $\ell=0$, we trivially have
  \(
    \| v_0 \|_{\ell^\infty}
    = \| z \|_{\ell^\infty}
    \leq R_0
    = R_0 \cdot (2 R)^\ell \cdot N_1 \cdots N_{\ell-1}
  \)
  and furthermore
  \(
    \| v_0 - w_0 \|_{\ell^\infty}
    = 0
    \leq R_0 \cdot N_1 \cdots N_{\ell-1} \cdot 4^\ell R^{\ell-1} \cdot \| \Phi - \Psi \|
    .
  \)

  Next, if the claimed estimates hold for some $\ell \in \{ 0,\dots, L-1\}$, we see
  \begin{align*}
    \big| (v_{\ell+1})_j \big|
    & = \big|
          \sigma_{\ell+1} \big( (B_{\ell+1} v_\ell + c_{\ell+1} )_j \big)
        \big|
      \leq \big|
             (B_{\ell+1} v_\ell + c_{\ell+1} )_j
           \big| \\
    & \leq |(c_{\ell+1})_j|
           + \sum_{m=1}^{N_\ell}
               |(B_{\ell+1})_{j,m}| \, |(v_\ell)_m|
      \leq R + N_\ell \, R \, R_0 \cdot (2 R)^\ell \cdot N_1 \cdots N_{\ell-1} \\
    & \leq R_0 \cdot (2 R)^{\ell+1} \cdot N_1 \cdots N_{\ell}
           \cdot \big(
                   \tfrac{1}{2} \tfrac{1}{R_0 \, (2 R)^\ell \, N_1 \cdots N_\ell}
                   + \tfrac{1}{2}
                 \big)
      \leq R_0 \cdot (2 R)^{\ell+1} \cdot N_1 \cdots N_{\ell} ,
  \end{align*}
  proving the first estimate for $\ell+1$ instead of $\ell$.
  In a similar way, we see
  \begin{equation}
    \begin{split}
      \big| (w_{\ell+1})_j - (v_{\ell+1})_j \big|
      & = \big|
            \sigma_{\ell+1} \big( (A_{\ell+1} w_\ell + b_{\ell+1})_j \big)
            - \sigma_{\ell+1} \big( (B_{\ell+1} v_\ell + c_{\ell+1})_j \big)
          \big| \\
      & \leq \bigl|(A_{\ell+1} \, w_\ell - B_{\ell+1} \, v_\ell)_j\bigr|
             + \| b_{\ell+1} - c_{\ell+1} \|_{\ell^\infty} .
    \end{split}
    \label{eq:RealizationLipschitzBoundStep1}
  \end{equation}
  Next, note that
  \begin{align*}
    \bigl|(A_{\ell+1} \, w_\ell - B_{\ell+1} \, v_\ell)_j\bigr|
    & \leq \sum_{m=1}^{N_\ell}
           \Big[
             \bigl|(A_{\ell+1})_{j,m} \big( (w_\ell)_m - (v_\ell)_m \big) \bigr|
             + \bigl|\big( (A_{\ell+1})_{j,m} - (B_{\ell+1})_{j,m}\big) (v_\ell)_m \bigr|
           \Big] \\
    & \leq N_\ell
           \cdot \Big(
                   R \cdot \| w_\ell - v_\ell \|_{\ell^\infty}
                   + \| \Phi - \Psi \| \cdot \| v_\ell \|_{\ell^\infty}
                 \Big) \\
    & \overset{(\ast)}{\leq}
           R_0 \cdot N_1 \cdots N_\ell \cdot 4^{\ell+1} R^\ell \cdot \| \Phi - \Psi \|
           \cdot \big(
                   \tfrac{1}{4}
                   + \tfrac{2^\ell}{4^{\ell+1}}
                 \big) \\
    & \leq \frac{1}{2} \cdot R_0 \cdot N_1 \cdots N_\ell \cdot 4^{\ell+1} R^\ell \cdot \| \Phi - \Psi \| ,
  \end{align*}
  where the step marked with $(\ast)$ used the induction hypothesis.
  Combining this estimate with \Cref{eq:RealizationLipschitzBoundStep1} and noting
  \(
    \| b_{\ell+1} - c_{\ell+1} \|_{\ell^\infty}
    \leq \| \Phi - \Psi \|
    \leq \frac{1}{2}
         \cdot R_0
         \cdot N_1 \cdots N_\ell
         \cdot 4^{\ell+1} R^\ell
         \cdot \| \Phi - \Psi \|
  \)
  completes the induction.
\end{proof}

\bibliographystyle{siamplain}

\bibliography{references}

\begin{thebibliography}{10}

\bibitem{ArenaNNInMultidimensionalDomains}
{\sc P.~Arena, L.~Fortuna, G.~Muscato, and M.~G. Xibilia}, {\em Neural networks
  in multidimensional domains: fundamentals and new trends in modelling and
  control}, vol.~234, Springer, 1998.

\bibitem{ArenaApproximationCapabilityOfComplexNeuralNetworks}
{\sc P.~Arena, L.~Fortuna, R.~Re, and M.~G. Xibilia}, {\em On the capability of
  neural networks with complex neurons in complex valued functions
  approximation}, in 1993 IEEE International Symposium on Circuits and Systems,
  IEEE, 1993, \url{https://doi.org/10.1109/ISCAS.1993.394188}.

\bibitem{ArenaMLPToApproximateComplexValuedFunctions}
{\sc P.~Arena, L.~Fortuna, R.~Re, and M.~G. Xibilia}, {\em Multilayer
  perceptrons to approximate complex valued functions}, International Journal
  of Neural Systems, 6 (1995), \url{https://doi.org/10.1142/s0129065795000299}.

\bibitem{BengioUnitaryEvolutionRNN}
{\sc M.~Arjovsky, A.~Shah, and Y.~Bengio}, {\em Unitary evolution recurrent
  neural networks}, in International Conference on Machine Learning, 2016,
  pp.~1120--1128.

\bibitem{BartlettVCBounds}
{\sc P.~L. Bartlett, N.~Harvey, C.~Liaw, and A.~Mehrabian}, {\em Nearly-tight
  {VC}-dimension and {P}seudodimension {B}ounds for {P}iecewise {L}inear
  {N}eural {N}etworks}, J. Mach. Learn. Res., 20 (2019), pp.~1--17.

\bibitem{BlackScholesGeneralizationError}
{\sc J.~Berner, P.~Grohs, and A.~Jentzen}, {\em Analysis of the generalization
  error: empirical risk minimization over deep artificial neural networks
  overcomes the curse of dimensionality in the numerical approximation of
  {B}lack-{S}choles partial differential equations}, SIAM J. Math. Data Sci., 2
  (2020), pp.~631--657, \url{https://doi.org/10.1137/19M125649X}.

\bibitem{BoelcskeiSparseNN}
{\sc H.~B\"{o}lcskei, P.~Grohs, G.~Kutyniok, and P.~Petersen}, {\em Optimal
  approximation with sparsely connected deep neural networks}, SIAM J. Math.
  Data Sci., 1 (2019), pp.~8--45, \url{https://doi.org/10.1137/18M118709X}.

\bibitem{ClementsEntropies}
{\sc G.~F. Clements}, {\em Entropies of several sets of real valued functions},
  Pacific J. Math., 13 (1963), pp.~1085--1095,
  \url{http://projecteuclid.org/euclid.pjm/1103034547}.

\bibitem{CybenkoUniversalApproximation}
{\sc G.~Cybenko}, {\em Approximation by superpositions of a sigmoidal
  function}, Mathematics of control, signals and systems, 2 (1989),
  pp.~303--314, \url{https://doi.org/10.1007/BF02551274}.

\bibitem{GlorotRectifierNetworks}
{\sc X.~Glorot, A.~Bordes, and Y.~Bengio}, {\em Deep sparse rectifier neural
  networks}, in Proceedings of the fourteenth international conference on
  artificial intelligence and statistics, 2011.

\bibitem{Gonon2019uniform}
{\sc L.~Gonon, P.~Grohs, A.~Jentzen, D.~Kofler, and D.~Šiška}, {\em Uniform
  error estimates for artificial neural network approximations for heat
  equations}, arXiv preprint arXiv:1911.09647,  (2019).

\bibitem{GrohsDNNHighDimensionalElliptic}
{\sc P.~Grohs and L.~Herrmann}, {\em {Deep neural network approximation for
  high-dimensional elliptic PDEs with boundary conditions}}, IMA Journal of
  Numerical Analysis,  (2021), \url{https://doi.org/10.1093/imanum/drab031}.

\bibitem{GrohsDNNHamiltonJacobiBellman}
{\sc P.~Grohs and L.~Herrmann}, {\em Deep neural network approximation for
  high-dimensional parabolic {Hamilton-Jacobi-Bellman} equations}, arXiv
  preprint arXiv:2103.05744,  (2021).

\bibitem{HiroseCliffordBook}
{\sc A.~Hirose}, {\em Complex-valued neural networks: theories and
  applications}, vol.~5, World Scientific, 2003,
  \url{https://doi.org/10.1142/5345}.

\bibitem{HornikUniversalApproximation}
{\sc K.~Hornik}, {\em Approximation capabilities of multilayer feedforward
  networks}, Neural Networks, 4 (1991), pp.~251--257,
  \url{https://doi.org/10.1016/0893-6080(91)90009-T}.

\bibitem{HornikStinchcombeWhiteUniversal}
{\sc K.~Hornik, M.~Stinchcombe, and H.~White}, {\em Multilayer feedforward
  networks are universal approximators}, Neural Networks, 2 (1989),
  pp.~359--366, \url{https://doi.org/10.1016/0893-6080(89)90020-8}.

\bibitem{KrizhevskyImagenet}
{\sc A.~Krizhevsky, I.~Sutskever, and G.~E. Hinton}, {\em Image{N}et
  classification with deep convolutional neural networks}, Communications of
  the ACM, 60 (2017), \url{https://doi.org/10.1145/3065386}.

\bibitem{LeCunDeepLearningNature}
{\sc Y.~LeCun, Y.~Bengio, and G.~Hinton}, {\em Deep learning}, Nature, 521
  (2015).

\bibitem{PinkusUniversalApproximation}
{\sc M.~Leshno, V.~Lin, A.~Pinkus, and S.~Schocken}, {\em Multilayer
  feedforward networks with a nonpolynomial activation function can approximate
  any function}, Neural Networks, 6 (1993), pp.~861--867,
  \url{https://doi.org/10.1016/S0893-6080(05)80131-5}.

\bibitem{LuReLUFiniteDepthUniform}
{\sc J.~Lu, Z.~Shen, H.~Yang, and S.~Zhang}, {\em Deep network approximation
  for smooth functions}, SIAM J. Math. Anal., 53 (2021), pp.~5465--5506,
  \url{https://doi.org/10.1137/20M134695X}.

\bibitem{MhaskarSmoothActivationOptimalRate}
{\sc H.~N. Mhaskar}, {\em Neural networks for optimal approximation of smooth
  and analytic functions}, Neural computation, 8 (1996),
  \url{https://doi.org/10.1162/neco.1996.8.1.164}.

\bibitem{PetersenOptimalApproximation}
{\sc P.~Petersen and F.~Voigtlaender}, {\em Optimal approximation of piecewise
  smooth functions using deep {ReLU} neural networks}, Neural Netw., 108
  (2018), \url{https://doi.org/10.1016/j.neunet.2018.08.019}.

\bibitem{SafranShamirArxivPaper}
{\sc I.~Safran and O.~Shamir}, {\em Depth-width tradeoffs in approximating
  natural functions with neural networks}, arXiv preprint arXiv:1610.09887,
  (2016).

\bibitem{SafranShamirICMLPaper}
{\sc I.~Safran and O.~Shamir}, {\em Depth-width tradeoffs in approximating
  natural functions with neural networks}, in International Conference on
  Machine Learning, PMLR, 2017, pp.~2979--2987.

\bibitem{SutskeverMachineTranslation}
{\sc I.~Sutskever, O.~Vinyals, and Q.~V. Le}, {\em Sequence to sequence
  learning with neural networks}, in Advances in neural information processing
  systems, 2014.

\bibitem{TrabelsiDeepComplexNetworks}
{\sc C.~Trabelsi, O.~Bilaniuk, Y.~Zhang, D.~Serdyuk, S.~Subramanian, J.~F.
  Santos, S.~Mehri, N.~Rostamzadeh, Y.~Bengio, and C.~J. Pal}, {\em Deep
  complex networks}, in ICLR, 2018, \url{https://openreview.net/forum?id =
  H1T2hmZAb}.

\bibitem{VershyninHighDimensionalProbablity}
{\sc R.~Vershynin}, {\em High-dimensional probability}, vol.~47 of Cambridge
  Series in Statistical and Probabilistic Mathematics, Cambridge University
  Press, Cambridge, 2018, \url{https://doi.org/10.1017/9781108231596}.

\bibitem{LustigComplexNNForMRI}
{\sc P.~{Virtue}, S.~X. {Yu}, and M.~{Lustig}}, {\em Better than real:
  {C}omplex-valued neural nets for {MRI} fingerprinting}, in 2017 IEEE
  International Conference on Image Processing (ICIP), 2017,
  \url{https://doi.org/10.1109/ICIP.2017.8297024}.

\bibitem{VoigtlaenderCVNNUniversality}
{\sc F.~Voigtlaender}, {\em The universal approximation theorem for
  complex-valued neural networks}, arXiv preprint arXiv:2012.03351,  (2020).

\bibitem{WolterComplexGatedRNNs}
{\sc M.~Wolter and A.~Yao}, {\em Complex gated recurrent neural networks}, in
  Advances in Neural Information Processing Systems, vol.~31, Curran
  Associates, Inc., 2018.

\bibitem{YarotskyReLUBounds}
{\sc D.~Yarotsky}, {\em Error bounds for approximations with deep {ReLU}
  networks}, Neural Networks, 94 (2017), pp.~103--114,
  \url{https://doi.org/10.1016/j.neunet.2017.07.002}.

\bibitem{YarotskyPhaseDiagram}
{\sc D.~Yarotsky and A.~Zhevnerchuk}, {\em The phase diagram of approximation
  rates for deep neural networks}, Advances in Neural Information Processing
  Systems, 33 (2020).

\end{thebibliography}

\end{document}